\documentclass[11pt]{amsart}
\usepackage{amsmath, amsfonts, amscd, epsfig, amssymb}

\newcommand{\bfO}{{\bf O}}
\newcommand{\btheta}{{\bar \theta}}

\newcommand{\RR}{{\mathbb R}}
\newcommand{\WW}{{\mathbb W}}

\newcommand{\ZZ}{{\mathbb Z}}

\newcommand{\CC}{{\mathbb C}}
\newcommand{\mA}{{\mathbb A}}

\newcommand{\FF}{{\mathbb F}}
\newcommand{\GG}{{\mathbb G}}

\newcommand\adots{\mathinner{\mkern2mu\raise1pt\hbox{.}
\mkern3mu\raise4pt\hbox{.}\mkern1mu\raise7pt\hbox{.}}}

\renewcommand{\div}{{\rm div}}

\newtheorem{theo}{Theorem}[section]
\newtheorem{prop}[theo]{Proposition}
\newtheorem{cor}[theo]{Corollary}
\newtheorem{lem}[theo]{Lemma}
\newtheorem{defi}[theo]{Definition}
\newtheorem{ass}[theo]{Assumption}
\newtheorem{assums}[theo]{Assumptions}

\newtheorem{exam}[theo]{Example}
\newtheorem{rem}[theo]{Remark}

\newtheorem{claim}[theo]{Claim}
\numberwithin{equation}{section}

\begin{document}

\title[Multidimensional scalar relaxation shocks]{Asymptotic Behavior of Multidimensional scalar Relaxation Shocks}
\author{Bongsuk Kwon and Kevin Zumbrun}

\date{\today}

\thanks{ This work was supported in part by the National Science Foundation award numbers DMS-0607721 and DMS-0300487.}

\address{Department of Mathematics, Indiana University, Bloomington, IN 47402}
\email{bkwon@indiana.edu}
\address{Department of Mathematics, Indiana University, Bloomington, IN 47402}
\email{kzumbrun@indiana.edu}

\begin{abstract}
We establish pointwise bounds for the Green function and consequent
linearized stability for multidimensional planar relaxation shocks
of general relaxation systems whose equilibrium model is scalar,
under the necessary assumption of spectral stability. Moreover, we
obtain nonlinear $L^{2}$ asymptotic behavior/sharp decay rate of
perturbed weak shocks of general simultaneously symmetrizable
relaxation systems, under small $L^{1}\cap H^{[d/2]+3}$
perturbations with first moment in the normal direction to the
front.

\end{abstract}

\maketitle \thispagestyle{empty}

\tableofcontents

\section{Introduction}

In this paper, we investigate the time-asymptotic stability of
multidimensional planar shocks of general relaxation systems whose
equilibrium model is scalar. \smallbreak Consider {\it hyperbolic
relaxation systems} of general form
\begin{equation}\label{main}
\begin{pmatrix}
u\\
v
\end{pmatrix}_t
+\sum_{j=1}^{d}
\begin{pmatrix}
f^{j}(u,v)\\
g^{j}(u,v)
\end{pmatrix}_{x_j}
=
\begin{pmatrix}
0\\ \tau^{-1}q(u,v)
\end{pmatrix}
\end{equation}
$u, f^{j}\in \RR^{1}$, $v,g^{j},q \in \RR^{r},$ with the condition
\begin{equation}\label{cond}
\text{Re } \sigma(q_v(u,v_{*}(u))<0
\end{equation}
along a smooth {\it equilibrium manifold} defined by
\begin{equation}\label{equi_mfold}
E:=\{(u,v)|q(u,v)\equiv 0\}
\end{equation}
and $\tau$  determines relaxation time. The first equation and the
second $r$ equations represent a conservation law for $u$ and
relaxation rate equations for $v$, respectively. The condition
\eqref{cond} implies that a perturbed solution eventually relaxes to
the equilibrium state.

\smallbreak The first equation for $u$ can be approximated by a
hyperbolic conservation law and a parabolic conservation law to
``zeroth order" and ``first order", respectively. Here, the
corresponding order is with respect to parameter $\tau$ determining
the relaxation time. To ``zeroth order", the corresponding
``relaxed" scalar equation is
\begin{equation}\label{zeroth}
u_t+\sum_{j=1}^{d} f^{j}_*(u)_{x_j}=0,
\end{equation} where $f^{j}_*(u):=f^{j}(u,v_*(u))$, and to ``first order", the corresponding parabolic conservation laws is
\begin{equation}\label{second}
u_t+\sum_{j=1}^{d} f^{j}_*(u)_{x_j}=(b^{jk}_*(u)u_{x_k})_{x_j},
\end{equation} where

\begin{align}\label{bjk}
b^{jk}_*=
\left\{%
\begin{array}{ll}
 - f^{j}_v q_v^{-1} (g^{j}_u-g^{j}_v q_v^{-1} q_u-( f^{j}_u-f^{j}_v q_v^{-1}
q_u)q_v^{-1} q_u) & ,\text{ if } j=k \\
 -\frac{1}{2}\Big(f^{j}_v q_v^{-1}(g^{j}_u-g^{j}_v q_v^{-1}
q_u+(f^{k}_u-f^{k}_v q_v^{-1} q_u) q_v^{-1} q_u)\\
\;\;\;\;\;\;
+f^{k}_v q_v^{-1}(g^{k}_u-g^{k}_v q_v^{-1} q_u+(f^{j}_u-f^{j}_v q_v^{-1} q_u) q_v^{-1} q_u)\Big) & ,\text{ if  } j\neq k \\
\end{array}%
\right.
\end{align}
is determined by the expansion of Fourier symbol(Chapmann-Enskog
expansion) as in Appendix \ref{appA}. \smallbreak These
approximation equations suggest us to investigate the existence of
shock wave solutions. A planar relaxation shock wave is a traveling
wave solution of \eqref{main} satisfying
\begin{equation}\label{profile}
\begin{aligned}
(u,v)(x,t)&=(\bar{u},\bar{v})(x_1-st),\\
\lim_{z\to \pm\infty}(\bar{u},\bar{v})(z)&=(u_\pm,v_\pm),
\end{aligned}
\end{equation}
where the end states $(u_\pm,v_\pm)$ satisfy $v^{*}(u_\pm)=v_\pm$
and $u_\pm$ is a shock solution of \eqref{zeroth}.

\smallbreak Such traveling wave solutions are known to exist for
small amplitude profiles, see for example, \cite{Liu,YoZ,MZ1}.
However, profiles of large amplitude may develop ``subshocks" or
jump discontinuities. We restrict here to the smooth and
small-amplitude case.

\smallbreak Stability of such multidimensional planar shock wave
solutions has been studied for specific models. Nonlinear stability
of planar shock fronts for the $3\times 3$ Jin-Xin model in two
spatial dimension has been proved in \cite{Li}. For a
two-dimensional shallow river model, Ha and Yu proved nonlinear
stability of small amplitude shocks in \cite{HY}. Both of these
analyses follow the approach introduced by Goodman in \cite{Go} to
treat the related scalar viscous case, based on energy estimates and
conservation of mass, yielding sup-norm convergence to the
unperturbed front with no rate. This method has since been greatly
sharpened in the scalar viscous case, using shock-tracking and
spectral (inverse Laplace-transform) methods and pointwise estimates
on the resolvent to obtain asymptotic behavior and sharp rates of
decay for general scalar models; see e.g., \cite{GM,
HoZ1,HoZ2}.\footnote{ Though the analysis of \cite{Go} also
involves an approximate front location, this is determined by a zero
residual mass condition convenient for energy estimates rather than
considerations of asymptotic behavior, and involves errors of the
same magnitude as the perturbation itself; see the discussions in
\cite{GM,HoZ1,HoZ2}. } However, up to now, no comparable result
has been carried out for the relaxation case.

In the present paper, generalizing the results of \cite{HoZ1,HoZ2}
in the viscous case, we prove stability, with asymptotic behavior
and sharp rates of decay, of small-amplitude multidimensional planar
relaxation shocks of $N\times N$ general systems $\eqref{main}$
whose equilibrium model is scalar, under the following assumptions.

\begin{assums}\label{ass:A}{ $ \,$ }
\medbreak
\textup{ (H0) } $f^{j}, g^{j}, q\in \mathcal{C}^{m+1}, m\geq
[d/2]+2$. \medbreak \textup{ (H1) }
(i)$\sigma\big(\sum_j\xi_j(df^{j},dg^{j})^{t}(u,v)\big)$ \textup{
real, semi-si1ultiplicity, for all
$\xi\in\RR^{d}$, and} (ii) $\big(\sigma(df^{1},dg^{1})^{t}(u,v)\big)$
\textup{ different from $s$.} \medbreak \textup{ (H2) }
$\sigma\big(\sum_j\xi_j df^{j}_*(u_\pm)\big) \textup{ real, distinct
and different from $s$.}$ \medbreak \textup{ (H3) }  $\Re\sigma\Big(i \sum_{j=1}^{d}
i\xi_j(df^{j},dg^{j})^{t}(u_{\pm},v_{\pm})-(0,dq)^{t}(u_{\pm},v_{\pm})\Big)\leq
-\theta|\xi|^{2}/(1+|\xi|^{2}) \textup{ for all } \xi \in \RR^{d}$ ,
$\theta>0$. \medbreak \textup{ (H4) } The set of solutions of
\eqref{main} forms a smooth manifold
$(\bar{u}_\delta,\bar{v}_\delta), \delta\in \mathcal{U}\in \RR^{1}$.
\end{assums}

Let $D(\lambda, \tilde \xi)$ as in Definition \ref{def:evans},
Section \ref{evans}, denote the Evans function associated with
Fourier transform $L_{\tilde \xi}$ of the linearized operator about
the wave, a function that is analytic in $\lambda$ for $\Re
\lambda\ge -\theta$, $\theta>0$, with zeros corresponding with
eigenvalue of $L_{\tilde \xi}$.
(For history and further discussion of the Evans function, see
\cite{AGJ, GZ, PZ} and references therein.)

\begin{assums}\label{strongass} (Strong spectral stability conditions)\smallbreak
\textup{ ($\mathcal{D}$1) } $D(\cdot,\tilde{\xi}) \textup{ has no
zeroes in } \{\Re\lambda\geq 0\} \textup{ except at }
\tilde{\xi}=\lambda=0$.
\smallbreak \textup{ ($\mathcal{D}$2) } $(d/d\lambda)D(0,0)\neq 0$.
\smallbreak \textup{ ($\mathcal{D}$3) } \textup{ A zero } $
\lambda_*(\tilde{\xi})$ \textup { of } $D(\cdot,\tilde{\xi})$
\textup{ satisfies } $\lambda_*(0)=0$ \textup{ and } $\Re\lambda_*(\tilde{\xi})\leq -\theta|\tilde{\xi}|^{2}$ \textup{ for }
$|\tilde{\xi}|$ \textup { sufficiently small. (Existence and local
uniqueness of $\lambda_*(\tilde{\xi})$ are guaranteed by the
Implicit Function Theorem and \textup {($\mathcal{D}$2)}.) }
\end{assums}

Sometimes it is more convenient to write \eqref{main} in the
abbreviated form
\begin{equation}\label{abb}
U_t+\sum_{j=1}^{d} A^{j}(U)U_{x_j}=\tau^{-1} Q(U)
\end{equation} where $U=(u,v)^{t}$, $A^{j}(U)=(df^{j},dg^{j})^{t}(u,v)$ and $Q(U)=(0,q(u,v))^{t}$.

\begin{assums}\label{further}{ $ \,$ }
\smallbreak \textup{($\mathcal{A}$1)} \eqref{abb} is symmetrizable
in the sense that there exists $A^{0}$ symmetric, positive definite
such that $A^{0}A^{j}$ are symmetric for all $j=1,2,...,d$ and
$A^{0}dQ$ is symmetric, negative semidefinite. \smallbreak
\textup{($\mathcal{A}$2)}  (Kawashima condition) There exists the
operator $K(\partial_x)$ such that
\begin{equation}\label{kaw0}
\widehat{K(\partial_x)f}(\xi)=i \bar{K}(\xi) \widehat{f}(\xi)
\end{equation}
where $\bar{K}(\xi)$ is a skew-symmetric operator which is smooth
and homogeneous degree one in $\xi$ satisfying

\begin{equation}
\Re\sigma\big(|\xi|^{2}A^{0}dQ-\sum_{j=1}^{d}\xi_j
\bar{K}(\xi)A^{j}\big)_{\pm} \leq -\theta |\xi|^{2} \textup{   for all }\xi \text{ in } \mathbb{R}^{d}\\
\end{equation}
\end{assums}

\begin{rem} 
\textup{
If $\eqref{abb}$ satisfies \textup{($\mathcal{A}$1)} and the
Genuine-Coupling condition that no eigenvector of $\sigma\Big(i
\sum_{j=1}^{d} \xi_j(df^{j},dg^{j})^{t}(u_{\pm},v_{\pm})\Big)$
lies in the kernel of $dQ(u_{\pm},v_{\pm})$, then \textup{($\mathcal{A}$2)} holds
\cite{K,SK,MZ5,Z4,Z5}. Moreover, conditions (A1)--(A2) imply
(H3) \cite{K,SK,Ze}.
}
\end{rem}

Conditions (H0)--(H4), (A1)--(A2) are the standard set of hypotheses
proposed by W.A. Yong for relaxation systems \cite{Yo}, as adapted
to the shock case by Mascia and Zumbrun \cite{MZ1}.  As described
in \cite{MZ1,MZ5}, (A1)--(A2) are satisfied for a wide variety of
relaxation systems, in which case all of (H0)--(H4) are satisfied
for sufficiently small-amplitude profiles under the single condition
(H1)(ii).

\smallbreak Before we state our main theorem, we briefly go over the
idea used by Goodman and Miller \cite{GM}
to give a formal qualitative description of the behavior of the
linear perturbation $ U(x,t):=\tilde{U}(x,t)-\bar{U}(x_1)
$, where $\tilde{U}(x,t)$ is a solution of \eqref{main} and $\bar{U}(x_1)$ is a shock wave solution. 
Linearizing \eqref{main} about $\bar{U}(x_1)$, we obtain the
linearized perturbation equations
\begin{equation}\label{linearized}
U_t=LU:=-\sum_{j=1}^{d}(A^{j}U)_{x_j}+(dQ) U
\end{equation}
where
\begin{equation}
A^{j}:=dF^{j}(\bar{U}(x_1)),\;\;\;\;\;\;\; dQ:=dQ(\bar{U}(x_1))
\end{equation}
depend only on the normal direction $x_1$. \smallbreak We can give a
heuristic approach to describe the behavior of the perturbation
$U=\tilde{U}-\bar{U}$. First, we approximate the operator
$e^{L_{\tilde{\xi}}t}$ by its formal spectrral projection
\begin{equation}
e^{\lambda_*(\tilde{\xi})t} \varphi\langle\tilde{\psi},
\hat{U_0}\rangle,
\end{equation} onto the top
eigenfunction of $L_{\tilde{\xi}}$, assuming a perturbation
expansion
\begin{align}\label{series1}
\lambda_*(\tilde \xi)&= \tilde{\gamma}^{1}\cdot \tilde{\xi} +
\tilde{\xi}^{t} \tilde{\gamma}^{2} \tilde{\xi} +\dots\\
&= i\tilde{\alpha}\cdot \tilde{\xi} -
\tilde{\xi}^{t} \tilde{\beta} \tilde{\xi} +\dots\nonumber
\end{align}
of the corresponding eigenvalue $\lambda_*$. Next we apply the
method of stationary phase to the inverse Fourier transform to
obtain the approximation

\begin{align}\label{sp}
U(x,t)&=e^{Lt}U_0= \frac{1}{(2\pi)^{d-1}}\int_{\mathbb{R}^{d-1}}
e^{L_{\tilde{\xi}}t}
e^{i\tilde{\xi}\cdot\tilde{x}}\hat{U_0} d\tilde{\xi}\\
&\sim \frac{1}{(2\pi)^{d-1}}\int_{\mathbb{R}^{d-1}}
 e^{\lambda_*(\tilde{\xi})t}
\varphi\langle\tilde{\psi}, \hat{U_0}\rangle
e^{i\tilde{\xi}\cdot\tilde{x}}d\tilde{\xi}\nonumber\\
&\sim -\frac{\bar U'(x_1)}{(2\pi)^{d-1}}\int_{\mathbb{R}^{d-1}}
 e^{i\tilde{\xi}\cdot\tilde{x}} e^{( i\tilde{\alpha}\cdot \tilde{\xi} -
\tilde{\xi}^{t} \tilde{\beta} \tilde{\xi})t} \hat{\delta}_0(\tilde{\xi}) d\tilde{\xi}\nonumber\\
&\sim -\bar{U}'(x_1)\delta(\tilde{x},t)\nonumber
\end{align}
where $\delta(\tilde{x},t)$ satisfies the transverse
convection-diffusion equation
\begin{equation}\label{conv-diff}
\delta_t+\tilde{\alpha}\cdot \nabla_{\tilde{x}}
\delta=\div_{\tilde{x}}(\tilde{\beta}\nabla_{\tilde{x}}\delta)
\end{equation} with initial data
\begin{equation}\label{conv-data}
\delta_0(\tilde{x})= \langle \tilde{\psi}, U_0(x)
\rangle_{L^{2}(x_1)} =-([u]^{-1},0) \int U_0 (x) dx_1.
\end{equation}

\smallbreak The following theorem shows that the formal linear
approximation $U(x,t)\sim -\delta(\tilde{x},t) \bar{U}'(x_1)$ is in
fact valid at the nonlinear level.

\begin{theo}\label{mainthm}
For fixed $U_-$, let $\bar U$ be a relaxation shock profile
\eqref{profile} satisfying (H0)--(H5), (A1)--(A2),
($\mathcal{D}$1)--($\mathcal{D}$3), with amplitude $|U_+-U_-|$
sufficiently small. If $| \tilde{U}_0-\bar{U}|_{L^{1}}$,$|
\tilde{U}_0-\bar{U}|_{L^{2}}$,
 $| x_1(\tilde{U}_0-\bar{U})|_{L^{1}}$, $| \tilde{U}_0-\bar{U}|_{H^{[d/2]+3}}
 \leq
\zeta_0$ sufficiently small, then for arbitrary small $\sigma>0$,
there holds
\begin{equation}\label{resultmain}
\big|\tilde{U}(x,t)-\bar{U}(x_1-\delta(\tilde{x},t))\big|
_{L^{2}(x)}\leq C\zeta_0(1+t)^{-(d-1)/4-1/2+\sigma}
\end{equation}
for dimensions $d\ge 2$, with $\delta$ as defined in
\eqref{conv-diff}--\eqref{conv-data} and $\tilde{\alpha}$, $\tilde{
\beta}$ as in \eqref{series1}. Moreover, the above result holds
with $\sigma=0$ for dimensions $d\geq3$.
\end{theo}

\subsection{Discussion and open problems}\label{discussion}

Previous results \cite{Li,HY} on multidimensional scalar relaxation
fronts rely on the structure of specific models and give only
stability without rates or behavior. Theorem \ref{mainthm} by
contrast applies to general equations, giving sharp decay rates and
a detailed picture of asymptotic behavior. On the other hand, it
relies on the assumption of spectral stability, which must be
verified.

We conjecture that conditions ($\mathcal{D}$1)--($\mathcal{D}$3)
might be verified by a singular perturbation argument like that of
\cite{PZ} for the one-dimensional (system) case. Verification by
this or other means is an important open problem. Another
interesting open problem would be to remove the restrictive
hypothesis (H1)(ii) as discussed in \cite{MZ5,MZ6}.

The main obstruction to the application to the relaxation problem of
the spectral techniques of \cite{GM,HoZ1,HoZ2} is to treat the
more singular high-frequency behavior associated with the hyperbolic
nature of the equations. In the viscous case, the linearized
operator about the wave is sectorial, generating an analytic
semigroup, and high-frequency contributions are essentially
negligible. In the relaxation case, the linearized operator
generates a $C^0$ semigroup, and there is substantial high-frequency
contribution. We conjecture that it is this difficulty that has so
far prevented application of these techniques, despite their
advantages of generality and detailed information on asymptotic
behavior.

This difficulty was overcome in the one-dimensional analysis of
\cite{MZ1} by direct calculation/detailed asymptotic expansion.
However, this would appear quite complicated to carry out in the
multi-dimensional case. Here, we follow instead a simplified version
of an approach suggested in \cite{Z4} in the context of
hyperbolic--parabolic systems, based on high-frequency energy
estimates. This is quite general, and should find application to
other problems with delicate high-frequency behavior; in particular,
it applies with little modification to the case of relaxation
equations whose equibrium models are systems, preparing the way for
a treatment of multidimensional shock fronts in this case, another
important open problem. We regard this method of treating high
frequencies as perhaps the main new contribution of this paper.

A point that should be mentioned is that the analysis of
\cite{HoZ1,HoZ2} was for arbitrary amplitude shock waves, whereas
the present analysis is limited to the small-amplitude case. The
reasons for this are two. First, there is a limitation already at
the level of the existence problem, since subshocks may form for
too-large amplitudes in general \cite{Liu}. However, supposing
existence of a sufficiently smooth shock profile, we face a
technical problem in carrying out the energy estimates of Section
\ref{sec:damping} in the presence of characteristics moving both to
the left and to the right.  This was overcome with great difficulty
in the one-dimensional case in \cite{MZ5}; we do not know how or
whether this is possible in multi-dimensions.


As a final open problem, we mention the treatment for relaxation
systems of existence and stability of relaxation shock layers in the
small-relaxation time limit, analogous to the small-viscosity
analysis of \cite{GMWZ} in the hyperbolic--parabolic case.

\section{Preliminaries}

We begin with a series of preparatory steps, loosely following
\cite{MZ1, Z4}.

\subsection{Spectral resolution formulae} We derive the spectral resolution formula by proving that the linearized operator about the wave generates $C^{0}$ semigroup.

Linearizing \eqref{main} about the wave $\bar{U}(x_1)$, we have the
linearized equations
\begin{equation}\label{le}
U_t=LU:=-\sum_j (\bar{A}^{j} U)_{x_j} +d\bar{Q}U,\;\;\;\;\;\;
U(0)=U_0,
\end{equation}
where $d\bar{Q}=dQ(\bar{U}(x_1)),\;\;\;
\bar{A}^{j}=dF^{j}(\bar{U}(x_1))$. Taking the Fourier transform in
the transverse directions $\tilde{x}:=(x_2,...,x_d)$, we reduce to a
family of partial differential equations (PDE)
\begin{equation}
\hat{U}_t=L_{\tilde{\xi}}\hat{U}:=-(\bar{A}^{1}U)'-i\sum_{j=2}^{d}
\xi_j\bar{A}^{j} \hat{U} +d\bar{Q}\hat{U},\;\;\;
\hat{U}(0)=\hat{U}_0
\end{equation}
in $(x_1,t)$ indexed by frequency $\tilde{\xi} \in \RR^{d-1}$, where
$\hat{U}=\hat{U}(x_1,\tilde{\xi},t)$ denotes the Fourier transform
of $U=U(x,t)$ in $\tilde{x}$ and $``\; '\; "$ denotes $d/dx_1$.
Taking the Laplace transform in $t$, we reduce to the resolvent
equation
\begin{equation}\label{resolvent}
(\lambda-L_{\tilde{\xi}})\hat{\hat{U}}=\hat{U}_0
\end{equation}
where $\hat{\hat{U}}(x_1,\tilde{\xi},\lambda)$ denotes the
Laplace-Fourier transform of $U=U(x,t)$.

\begin{defi} \textup{ (a) The Green function $G(x,t;y)$ associated with the linearized equations $\eqref{le}$ is defined by}\\
\;\;\;\;\;\;\;\;\;\;\;\;\;\; \textup{ (i) } $(\partial_t-L_{\tilde{\xi}})G=0$ \textup{ in the distributional sense, for all $t>0$} and,\\
\;\;\;\;\;\;\;\;\;\;\;\;\;\; \textup{ (ii) }
$G(x,t;y)\rightharpoondown \delta(x-y)$ \textup{ as } $t\rightarrow
0$. \smallbreak \textup{ (b) } \textup{ The resolvent kernel
$G_{\lambda,\tilde{\xi}}(x_1,y_1)$ associated with the resolvent
equation $\eqref{resolvent}$ is defined as a distributional solution
of 
\begin{equation}
(\lambda-L_{\tilde{\xi}})G_{\lambda,\tilde{\xi}}(x_1,y_1)=\delta(x-y).
\end{equation}
Formally, one can write
\begin{equation}
G(x,t;y):=e^{Lt}\delta(x-y)
\end{equation} and
\begin{equation}
G_{\lambda,\tilde{\xi}}(x_1,y_1):=(\lambda-L_{\tilde{\xi}})^{-1}\delta(x_1-y_1)
\end{equation}
}
\end{defi}

\begin{prop}[\cite{MZ1, Z4}]
Under assumptions (H0)--(H4), (A1)--(A2), $L$ generates a
$\mathcal{C}^{0}$ semigroup $|e^{Lt}|\leq C e^{\eta_0 t}$ on $L^{2}$
with domain $\mathcal{D}(L):=\{U: U, LU \in L^{2}\}$, satisfying the
generalized spectral resolution formula, for some $\eta>\eta_0$,

\begin{equation}\label{spectral}
G(x,t;y)=
\frac{1}{(2\pi i)^{d}}
\textup{ P.V. } \int_{\eta-i\infty}^{\eta+i\infty}
\int_{\RR^{d-1}} e^{i\tilde{\xi}\cdot \tilde{x}+\lambda t}
G_{\lambda,\tilde{\xi}} (x_1,y_1) d\tilde{\xi} d\lambda,
\end{equation}
\end{prop}
\begin{proof}
Performing an elementary energy estimate, under the assumption of
symmetrizability of \eqref{abb}, we establish that
\begin{equation}\label{garding}
|A^{0}U|_{L^{2}}\leq
|\lambda-\lambda_*|^{-1}|A^{0}(L-\lambda)U|_{L^{2}}
\end{equation}
for a symmetrizer $A^{0}$ and all $U \in \mathcal{D}(L)$ and real
$\lambda$ greater than some value $\lambda_*$. If, in addition,
$A^{j}$ and $Q$ are asymptotically constant as
$x_1\rightarrow\pm\infty$, then it is shown that $L$ generates a
$C^{0}$ semigroup $e^{Lt}$ on $L^{2}$, satisfying
$|e^{Lt}|_{L^{2}}\leq C e^{\omega t}$ for some real $\omega$
\cite{MZ1}. This is done by \eqref{garding} and a standard result
of Henry. Therefore, by \cite{Pa}, p.1., the inverse Laplace-Fourier
Transform formula holds for $L, e^{Lt}$:
\begin{equation}
e^{Lt}f=
\frac{1}{(2\pi i)^{d}}
\textup{ P.V. } \int_{\eta-i\infty}^{\eta+i\infty}
\int_{\RR^{d-1}} e^{i\tilde{\xi}\cdot \tilde{x}+\lambda t}
(L_{\tilde{\xi}}-\lambda)^{-1} d\tilde{\xi} d\lambda,
\end{equation}
As a consequence, \eqref{spectral} holds in the distribution sense.
\end{proof}
\subsection{The Evans function}\label{evans}
Consider the homogeneous eigenvalue equation
\begin{equation}\label{ev}
(\lambda-L_{\tilde{\xi}})W=\\
(\lambda+i\sum_{j=2}^{d}\xi_jA^{j}-dQ)W+(A^{1}W)'=0
\end{equation}
and its constant-coefficient limits as $x_1\to \pm \infty$,
\begin{equation}\label{const1}
(L_{\pm,\tilde{\xi}}-\lambda)W=
(dQ_{\pm}-A_{\tilde{\xi},\pm}-\lambda)W-(A^{1}_{\pm}W)'=0
\end{equation}
or, alternatively,
\begin{equation}\label{const2}
W'= (A^{1}_{\pm})^{-1} \Big(dQ_{\pm}-\sum_{j=2}^d i\xi_j
A_{j,\pm}-\lambda\Big) W.
\end{equation}

\begin{defi}\label{consplitdef}
The domain of consistent splitting $\Lambda$ is defined as the
connected component of $(\lambda, \tilde \xi)\in \CC\times \RR^{d-1}$
containing $\tilde \xi=0$ and $\lambda$ going to real $+\infty$, for
which the coefficients
\begin{equation}\label{coeff1}
(A^{1}_{\pm})^{-1} \Big(dQ_{\pm}-\sum_{j=2}^d i\xi_j
A_{j,\pm}-\lambda\Big)
\end{equation}
in \eqref{const2} have $k$ eigenvalues of negative real part and
$N-k$ eigenvalues of positive real part, with no pure imaginary
eigenvalues.
\end{defi}

\begin{lem}
Under assumptions \textup{ (H0)-(H1), (H3)},
\begin{equation}\label{Lambdabd}
\Lambda \subset \{(\lambda, \tilde \xi):\,
 \Re \lambda > -\theta|\tilde \xi|^{2}/(1+|\tilde \xi|^{2}) .
\end{equation}
In particular, for $|(\lambda, \tilde \xi)|\ge r>0$, arbitrary,
$\Lambda\subset \{\lambda:\, \Re \lambda \ge -\eta\}\times
\RR^{d-1}$, where $\eta(r):=\theta r^2>0$.
\end{lem}

\begin{proof}
Noting that eigenvalues $\mu(\lambda, \tilde \xi)$ of coefficient
\eqref{coeff1} relate to solutions of the dispersion-relation
$$
\lambda(\xi)\in \sigma \Big(dQ_{\pm}-\sum_{j=1}^d i\xi_j
A_{j,\pm}-\lambda\Big)
$$
by the relation $\mu=i\xi_1$, we find by Assumption (H3) that the
coefficient has no pure imaginary eigenvalues when $ \Re \lambda >
-\theta|\xi|^{2}/(1+|\xi|^{2})  $ for $\xi=(\xi_1, \tilde \xi)$,
all $\xi_1\in \RR$, or equivalently $ \Re \lambda > -\theta|\tilde
\xi|^{2}/(1+|\tilde \xi|^{2})  $. A straightforward homotopy
argument taking $\lambda$ to real plus infinity then gives the
result; see \cite{MZ1,Z3,Z4}.
\end{proof}

\begin{prop}\label{2.3}
Under assumptions \textup{ (H0)-(H1), (H3)}, for $(\lambda, \tilde
\xi)$ in the domain of consistent splitting $\Lambda$, there are
solutions of $\eqref{ev}$
\begin{equation}
\{\varphi^{+}_1(x_1;\lambda,\tilde{\xi}),\cdot\cdot\cdot,\varphi^{+}_k(x_1;\lambda,\tilde{\xi})\}\end{equation}
and
\begin{equation}
\{\varphi^{-}_{k+1}(x_1;\lambda,\tilde{\xi}),\cdot\cdot\cdot,\varphi^{-}_N(x_1;\lambda,\tilde{\xi})\},\end{equation}
$N=r+1$, which are locally analytic (in $(\lambda, \tilde \xi)$)
bases for the stable and unstable manifolds as $x_1\to +\infty$ and
$x_1\to -\infty$, respectively, that is, the (unique) manifolds of
solutions decaying exponentially as $x_1\to \pm \infty$.  There are
also solutions of $\eqref{ev}$
\begin{equation}
\{\psi^{-}_1(x_1;\lambda,\tilde{\xi}),\cdot\cdot\cdot,\psi^{-}_k(x_1;\lambda,\tilde{\xi})\}\end{equation}
and
\begin{equation}
\{\psi^{+}_{k+1}(x_1;\lambda,\tilde{\xi}),\cdot\cdot\cdot,\psi^{+}_N(x_1;\lambda,\tilde{\xi})\}\end{equation}
which are locally analytic (in $(\lambda, \tilde \xi)$) bases for
stable and unstable manifolds as $x_1\to -\infty$ and $x_1\to
+\infty$, respectively, that is, manifolds of solutions blowing up
exponentially as $x_1\to \pm \infty $ (not unique).
\end{prop}

\begin{proof}
This standard result holds for general variable-coefficient systems
whose coefficients converge exponentially as $x_1\to \pm \infty$ (a
consequence of the gap and conjugation lemmas of Appendix
\ref{appB}; see \cite{MZ1,MZ3,Z3,Z4}.
\end{proof}

%

\begin{defi}\label{def:evans}(Evans function)
For $(\lambda, \tilde \xi)$ in the domain of consistent splitting,
we define the Evans function as
\begin{equation}
D(\lambda,\tilde{\xi}):= \textup{ det
}(\varphi^{+}_1,...,\varphi^{+}_k,\varphi^{-}_{k+1},...,\varphi^{-}_N)|_{x_1=0}.
\end{equation}
\end{defi}

Evidently, the Evans function is locally analytic in $(\lambda,
\tilde \xi)$ in the domain of consistent splitting, with zeros of
$D(\cdot, \tilde \xi)$ corresponding to eigenvalues of $L_{\tilde
\xi}$. We shall show in Section \ref{subLF} that $\varphi_k^\pm$,
hence $D$ as well, extend analytically to a domain
$$
(\lambda, \tilde \xi)\in \{\lambda:\, \Re \lambda \ge -\eta\} \times
\RR^{d-1},\quad \eta>0.
$$

\subsection{Construction of the resolvent kernel}
We next derive explicit representation formulae for the resolvent
kernel $G_{\lambda,\tilde{\xi}}$.
We seek a solution of form\\
\begin{align}
G_{\lambda,\tilde{\xi}}(x_1,y_1)=
\left\{%
\begin{array}{ll}
 \Phi^{+}(x_1;\lambda,\tilde{\xi})N^{+}(y_1;\lambda,\tilde{\xi}) & ,\;\;\;x_1>y_1 \\
 \Phi^{-}(x_1;\lambda,\tilde{\xi})N^{-}(y_1;\lambda,\tilde{\xi}) & ,\;\;\;x_1<y_1 \\
\end{array}%
\right.
\end{align}
where
\begin{equation}
\Phi^{+}(x_1;\lambda,\tilde{\xi})=
\big(\varphi^{+}_1(x_1;\tilde{\xi},\lambda),\cdot\cdot\cdot,\varphi^{+}_k(x_1;\tilde{\xi},\lambda)\big)\in
\mathbb{R}^{N\times k}
\end{equation} and
\begin{equation}
\Phi^{-}(x_1;\lambda,\tilde{\xi})=
\big(\varphi^{-}_{k+1}(x_{1};\tilde{\xi},\lambda),\cdot\cdot\cdot,\varphi^{-}_N(x_1;\tilde{\xi},\lambda)\big)\in
\mathbb{R}^{N\times(N-k)}.
\end{equation}

Imposing the jump condition \begin{equation}\label{jump}
\begin{pmatrix}
\Phi^{+}(y_1;\lambda,\tilde{\xi}) &
\Phi^{-}(y_1;\lambda,\tilde{\xi})
\end{pmatrix}
\begin{pmatrix}
N^{+}(y_1;\lambda,\tilde{\xi}) \\
-N^{-}(y_1;\lambda,\tilde{\xi})
\end{pmatrix}=-(A^{1})^{-1}(y_1),
\end{equation}
 and inverting \eqref{jump}, we express for the resolvent kernel
$G_{\lambda,\tilde{\xi}}$:

\begin{align}
G_{\lambda,\tilde{\xi}}(x_1,y_1)=
\left\{%
\begin{array}{ll}
 -\big(
 \Phi^{+}(x_1;\lambda,\tilde{\xi}) \;\;\; 0\big)
 \big(\Phi^{+} \;\;\; \Phi^{-}\big)^{-1}(y_1;\lambda,\xi)(A^{1})^{-1}(y_1),& \\
  \:\:\:\:\:\:\:\:\:\:\:\:\:\:\:\:\:\:\:\:\:\:\:\:\:\:\:\:\:\:\:\:\:\:\:\:\:\:\:\:\:\:\:\:\:
 \:\:\:\:\:\:\:\:\:\:\:\:\:\:\:\:\:\:\:\:\:\:\:\:\:\:\:\:\:\:\:\:\:\:\:\:\:\:\:\:\:\:\:\:\: x_1>y_1 \\
 \big(
 0\;\;\; \Phi^{-}(x_1;\lambda,\tilde{\xi})\big)
 \big(
 \Phi^{+} \;\;\; \Phi^{-}\big)^{-1}(y_1;\lambda,\xi)(A^{1})^{-1}(y_1),&\\
 \:\:\:\:\:\:\:\:\:\:\:\:\:\:\:\:\:\:\:\:\:\:\:\:\:\:\:\:\:\:\:\:\:\:\:\:\:\:\:\:\:\:\:\:\:
 \:\:\:\:\:\:\:\:\:\:\:\:\:\:\:\:\:\:\:\:\:\:\:\:\:\:\:\:\:\:\:\:\:\:\:\:\:\:\:\:\:\:\:\:\: x_1<y_1 \\
\end{array}%
\right.
\end{align}



Now, consider the dual equation of \eqref{ev}.
\begin{equation}
(L^{*}_{\tilde{\xi}}-\lambda^{*})\tilde{W}=0
\end{equation}
where
$$L^{*}_{\tilde{\xi}}\tilde{W}:=(A^{1})^{*}\tilde{W}'+(dQ^{*}-A_{\tilde{\xi}}^{*})\tilde{W}$$
\begin{lem}
For any $W, \tilde{W}$ such that $(L_{\tilde{\xi}}-\lambda)W=0$ and
$(L^{*}_{\tilde{\xi}}-\lambda^{*})\tilde{W}=0$, there holds
\begin{equation}\label{dual}
\langle\tilde{W},A^{1}W\rangle\equiv constant
\end{equation}
\end{lem}
\begin{proof}
\begin{align}
\langle\tilde{W},A^{1}W\rangle'&=\langle(A^{1})^{*}\tilde{W}',W\rangle+\langle\tilde{W},(A^{1}W)'\rangle\nonumber\\
&=\langle(\lambda^{*}I-dQ^{*}+A_{\tilde{\xi}}^{*})\tilde{W},W\rangle+\langle\tilde{W},(-\lambda I+dQ-A_{\tilde{\xi}})W\rangle\nonumber\\
&=0\nonumber
\end{align}
\end{proof}

From \eqref{dual}, it follows that if there are $k$ independent
solutions $\varphi^{+}_1,...,\varphi^{+}_k$ of
$(L_{\tilde{\xi}}-\lambda I)W=0$ decaying at $+\infty$ and $N-k$
independent solutions $\varphi^{-}_{k+1},...,\varphi^{-}_N$ of the
same equation decaying at $-\infty$, then there exist $N-k$
independent solutions
$\tilde{\psi}^{+}_{k+1},...,\tilde{\psi}^{+}_N$ of
$(L_{\tilde{\xi}}^{*}-\lambda^{*}I)\tilde{W}=0$ decaying at
$+\infty$ and $k$ independent solutions
$\tilde{\psi}^{-}_1,...,\tilde{\psi}^{-}_k$ of the same equation
decaying at $-\infty$. More precisely, setting
\begin{equation}
\Psi^{+}(x_1;\lambda,\tilde{\xi})=\big(\psi^{+}_{k+1}(x_1;\lambda,\tilde{\xi})
\;\;\;\cdot\cdot\cdot\;\;\;\psi^{+}_{N}(x_1;\lambda,\tilde{\xi})\big)\in
\RR^{N\times (N-k)}
\end{equation}
\begin{equation}
\Psi^{-}(x_1;\lambda,\tilde{\xi})=\big(\psi^{-}_{1}(x_1;\lambda,\tilde{\xi})
\;\;\;\cdot\cdot\cdot\;\;\;\psi^{-}_{k}(x_1;\lambda,\tilde{\xi})\big)\in
\RR^{N\times k}
\end{equation}and
\begin{equation}
\Psi(x_1;\lambda,\tilde{\xi})=\big(\Psi^{-}(x_1;\lambda,\tilde{\xi})
\;\; \; \Psi^{+}(x_1;\lambda,\tilde{\xi})\big)\in \RR^{N\times N}
\end{equation}
where $\psi^{\pm}_j$ are the exponentially growing solutions at
$\pm\infty$ respectively, of $(L_{\tilde{\xi}}-\lambda I)W=0$ as
described above. We may define dual exponentially decaying and
growing solutions $\tilde{\psi}^{\pm}_j$ and
$\tilde{\varphi}^{\pm}_j$ via
\begin{equation}
\big(\tilde{\Psi}\;\;\;\tilde{\Phi}\big)^{*}_{\pm}A^{1}\big(\Psi\;\;\;\Phi\big)_{\pm}\equiv
I
\end{equation}
We seek the Green function $G_{\lambda,\tilde{\xi}}$ in the form
\begin{align}\label{G1}
G_{\lambda,\tilde{\xi}}(x_1,y_1)=
\left\{%
\begin{array}{ll}
 \Phi^{+}(x_1;\lambda,\tilde{\xi})\; M^{+}(\lambda,\tilde{\xi})\;\tilde{\Psi}^{-*}(y_1;\lambda,\tilde{\xi}) & ,\;\;\; x_1>y_1 \\
 \Phi^{-}(x_1;\lambda,\tilde{\xi})\; M^{-}(\lambda,\tilde{\xi})\;\tilde{\Psi}^{+*}(y_1;\lambda,\tilde{\xi}) & ,\;\;\;x_1<y_1 \\
\end{array}%
\right.,
\end{align}
where
\begin{equation}\label{M}
M(\lambda,\tilde{\xi}):=\begin{pmatrix}
-M^{+}(\lambda,\tilde{\xi}) & 0\\
0 &
M^{-}(\lambda,\tilde{\xi})\end{pmatrix}=\Phi^{-1}(z;\lambda,\tilde{\xi})\;(A^{1})^{-1}(z)\;\tilde{\Psi}^{-1*}(z;\lambda,\tilde{\xi})
\end{equation}
and
\begin{equation}
\tilde{\Psi}:=\big(\tilde{\Psi}^{-}\;\;\;\tilde{\Psi}^{+}\big).
\end{equation}
Note that the independence of the righthand side with respect to $z$
is a consequence of the previous lemma. Thus,

\begin{align}\label{GM}
G_{\lambda,\tilde{\xi}}(x_1,y_1)=
\left\{%
\begin{array}{ll}
 -\left(
 \Phi^{+}(x_1;\lambda,\tilde{\xi}) \;\;\; 0\right)\;
 M(\lambda,\tilde{\xi})\;\left(\tilde{\Psi}^{-}\left(y_1;\lambda,\xi\right)\;\;\; 0\right)^{*}& ,\;\;\;x_1>y_1 \\
 \left(0\;\;\;
 \Phi^{-}(x_1;\lambda,\tilde{\xi}) \right)\;
 M(\lambda,\tilde{\xi})\;\left(0\;\;\; \tilde{\Psi}^{+}\left(y_1;\lambda,\xi\right)\right)^{*} & ,\;\;\; x_1<y_1 \\
\end{array}%
\right.
\end{align}

\begin{prop}
For $(\lambda,\tilde \xi)\in \Lambda$, there hold
\begin{equation}\label{e:101a}\begin{array}{rl}
G_{\lambda,\tilde{\xi}} (x_1,y_1) = & \left\{
\begin{array}{ll}
\vspace{2mm} \displaystyle
\sum_{k,j}M^{+}_{jk}(\lambda,\tilde{\xi})\varphi^{+}_j(x_1;\lambda,\tilde{\xi})\tilde{\psi}_k^{-}(y_1;\lambda,\tilde{\xi})^{*}
& \text{ for } y_1\leq 0 \leq x_1,\\
\vspace{2mm} \displaystyle \sum_{k,j}d^{+}_{jk}(\lambda,\tilde{\xi})\varphi^{-}_j(x_1;\lambda,\tilde{\xi})\tilde{\psi}_k^{-}(y_1;\lambda,\tilde{\xi})^{*} &\\
\vspace{2mm} \hspace{8mm} \displaystyle -\sum_{k}\psi^{-}_k(x_1;\lambda,\tilde{\xi})\tilde{\psi}^{-}_k(y_1;\lambda,\tilde{\xi})^{*} & \text{ for }y_1\leq x_1\leq 0,\\
\vspace{2mm} \displaystyle \sum_{k,j}d^{-}_{jk}(\lambda,\tilde{\xi})\varphi^{-}_j(x_1;\lambda,\tilde{\xi})\tilde{\psi}_k^{-}(y_1;\lambda,\tilde{\xi})^{*} &\\
\hspace{8mm} \displaystyle
+\sum_{k}\varphi^{-}_k(x_1;\lambda,\tilde{\xi})\tilde{\varphi}^{-}_k(y_1;\lambda,\tilde{\xi})^{*}
& \text{ for }x_1\leq y_1\leq 0,\\
\end{array}
\right.
\end{array}
\end{equation}
with
\begin{equation}M^{+}=(-I,0)\begin{pmatrix}
\Phi^{+}&\Phi^{-}\end{pmatrix}^{-1}\Psi^{-}
\end{equation}
and
\begin{equation}d^{+}=(0,I)\begin{pmatrix}
\Phi^{+}&\Phi^{-}\end{pmatrix}^{-1}\Psi^{-}.
\end{equation}
A symmetric representation holds for $y_1\geq 0$.
\end{prop}

\begin{proof}
This follows exactly as in the one-dimensional case \cite{MZ1}, by
\eqref{M} together with Kramer's rule.
\end{proof}

\begin{rem}\label{genbd}
\textup{ Representation \eqref{GM} together with uniform exponential
decay of $\Phi^\pm$, $\tilde \Psi^\pm$, Proposition \ref{2.3}, and
the fact that $d_\pm$ are bounded when the Evans function
$D:=\det(\Phi^+,\Phi^-)$ does not vanish yields uniform bounds
$$
|G_{\lambda, \tilde \xi}(x,y)|\le Ce^{-\theta |x-y|},
$$
$\theta>0$, on the resolvent set $\rho(L_{\tilde \xi})$, in
particular (by assumption ($\mathcal{D}1$)) for $\Re \lambda \ge
-\eta$, $\eta>0$ on intermediate frequencies $1/R\le
|(\lambda,\tilde \xi)|\le R$, $R>0$ arbitrary. However, we shall not
use this in our analysis, carrying out instead energy-based
resolvent estimates for intermediate and high frequencies. We shall
use \eqref{GM} only in the {\it low frequency} regime
$|(\lambda,\tilde \xi)|<<1$. }
\end{rem}

\subsection{Low frequency bounds}\label{subLF}

We now examine in further detail behavior for small frequencies,
carrying out at the same time the analytic extension of normal modes
and Evans function beyond the region of consistent splitting. As in
\cite{HoZ2}, for our later arguments it will be important to
enlarge the domain of $\tilde \xi$ to complex values $\tilde \xi\in
\CC^{d-1}$, and so we will do this at the same time.

\begin{lem}
Under the assumptions of Theorem \ref{main}, for $|(\tilde \xi,
\lambda)|$ sufficiently small, $\tilde \xi$ now taken in
$\CC^{d-1}$, the eigenvalue equation
$(L_{\pm,\tilde{\xi}}-\lambda)W=0$ associated with the limiting,
constant-coefficient operator $L_{\pm,\tilde{\xi}}$ has a basis of
$(1+r)$ solutions, for $m=1,...,r$,
\begin{equation}
\bar{W}^{\pm}_m=e^{\mu_m^{\pm}(\lambda,\tilde{\xi})x_1}V^{\pm}_m(\lambda,\tilde{\xi}),
\end{equation}
$\mu^{\pm}_m, V^{\pm}_m$, analytic in $\lambda$ and $\tilde{\xi}$,
consisting of r ``fast" modes
\begin{equation}\label{mus}
\mu^{\pm}_m=\gamma^{\pm}_m+
\mathcal{O}(\lambda,\tilde{\xi}),\;\;\;\;\;V^{\pm}_m=(A^{1}_\pm)^{-1}
S_m^{\pm}+\mathcal{O}(\lambda,\tilde{\xi}),
\end{equation}
$S^{\pm}_m=(0,s^{\pm t}_m)^{t}$ where $\gamma^{\pm}_m, s^{\pm}_m$
are eigenvalues and associated right eigenvectors of $dq_\pm
(A^{1}_\pm)^{-1}(0,I_r)^{t}$
(equivalently, $\gamma^{\pm}_m, S^{\pm}_m$ are nonzero eigenvalues
and associated right eigenvectors of $(Q_\pm-i \sum_{j\neq 1}\xi_j
A^{j}_\pm)(A^{1}_\pm)^{-1}$), and $1$ ``slow" mode
\begin{align}\label{mr+1}
\mu^{\pm}_{r+1}(\lambda,\tilde{\xi})
&=(-\frac{1}{a_1^{\pm}})(\lambda+i\tilde{\xi}\cdot\tilde{a}^\pm)+\frac{b_{11}^{*,\pm}}{(a_1^{\pm})^{3}}(\lambda+i\tilde{\xi}\cdot\tilde{a}^{\pm})^{2}\\
&\;\;\;\;\;
-\frac{1}{a_1^{\pm}}\tilde{\xi}^{t}B_{11}^{*}\tilde{\xi}+\mathcal{O}(|\lambda+|\tilde{\xi}||^{3})\nonumber
\end{align}

and
\begin{equation}
V^{\pm}_{r+1}(\lambda,\tilde{\xi}):=R^{*\pm}_1+\mathcal{O}(\lambda,\tilde{\xi}),
\end{equation}
where $a_1^\pm=a_1^*(\pm \infty)$, $\tilde a^\pm=(a_2^*,\dots,
a_d^*)(\pm \infty)$, $R^{*\pm}_1=V^1_1(\pm \infty)$, and
$b_{jk}^{*,\pm}=b_{jk}^*(\pm\infty)$, with $a_j^{*}$, $V^1_1$, and
$b_{jk}^*$ as defined in \eqref{ajstar}, \eqref{Vj1}, and
\eqref{Abjk}, Appendix \ref{appA}. 
Likewise, the adjoint eigenvalue equation
$(L_{\pm,\tilde{\xi}}-\lambda)^{*}Z=0$ has a basis of solutions
\begin{equation}
\bar{\tilde{W}}^{\pm}_m=e^{-\mu_m^{\pm}(\lambda,\tilde{\xi})x_1}\tilde{V}^{\pm}_m(\lambda,\tilde{\xi}),
\end{equation} where
\begin{equation}
\tilde{V}^{\pm}_m(\lambda,\tilde{\xi})=\tilde{T}^{\pm}_m+\mathcal{O}(\lambda,\tilde{\xi})
\end{equation}
and
\begin{equation}
\tilde{V}^{\pm}_{r+1}(\lambda,\tilde{\xi})=\tilde{L}^{\pm
*}_1+\mathcal{O}(\lambda,\tilde{\xi}),
\end{equation}
where $\tilde{V}$ is analytic in $(\lambda, \tilde{\xi})$,
$\tilde{T}_m$ are the left eigenvectors of
$$
(Q_\pm-i \sum_{j\neq 1}\xi_j A^{j}_\pm)(A^{1}_\pm)^{-1}
$$
associated with the nonzero eigenvalues $-\mu_m^{\pm}$,
and $L^{*\pm}_1=(1,0_{r})^T$.
\end{lem}


\begin{proof}
By the inversion of the expressions
\begin{align}
\lambda(\xi)&=-i\xi\cdot a^{*}-\xi^{t} B^{*} \xi+\cdot\cdot\cdot
\;\;\;\\
&= -i \xi_1 a_1 -(b_{11}^{*}\xi^{2}_{1}+\sum_{j\neq 1}b_{j1}^{*}\xi_1\xi_j+\sum_{k\neq 1}b_{1k}^{*}\xi_k\xi_1)-i \tilde{\xi}\cdot \tilde{a}-\tilde{\xi}^{t}B_{11}^{*}\tilde{\xi}+\cdot\cdot\cdot\nonumber\\
&= -i \xi_1 \big(a_1-i( \sum_{j\neq 1}b_{j1}^{*}\xi_j+\sum_{k\neq 1}b_{1k}^{*}\xi_k)\big)-b_{11}^{*}\xi^{2}_{1}-i \tilde{\xi}\cdot \tilde{a}-\tilde{\xi}^{t}B_{11}^{*}\tilde{\xi}+\cdot\cdot\cdot \nonumber\\
&= -i \xi_1 \big(a_1- 2i \sum_{j\neq 1}b_{j1}^{*}\xi_j
\big)-b_{11}^{*}\xi^{2}_{1}-i \tilde{\xi}\cdot
\tilde{a}-\tilde{\xi}^{t}B_{11}^{*}\tilde{\xi}+\cdot\cdot\cdot
\nonumber
\end{align}
carried out in Appendix \ref{appA} for the dispersion curves near
$\xi=0$, together with the fundamental relation $\mu=i \xi_1$, we
have $R^{*\pm}_1=V^1_1(\pm \infty)$ and
\begin{eqnarray}
\mu^{\pm}_{r+1}(\lambda,\tilde{\xi})&:=&-\frac{\lambda}{a^{\pm}_1
-2i\sum_{j\neq 1}b_{j1}^{*,\pm}\xi_j}
+\frac{b^{*,\pm}_{11}}{(a^{\pm}_1 -2i\sum_{j\neq
1}b_{j1}^{*,\pm}\xi_j
)^{3} }\lambda^{2}+ \mathcal{O}(\lambda^{3})\nonumber\\
&=&-\frac{1}{a_1^{\pm}}(\lambda+i\tilde{\xi}\cdot\tilde{a}^{\pm}+\tilde{\xi}^{t}B_{11}^{*,\pm}\tilde{\xi})+\frac{b_{11}^{*,\pm}}{(a_1^{\pm})^{3}}(\lambda+i\tilde{\xi}\cdot\tilde{a}^{\pm}+\tilde{\xi}^{t}B_{11}^{*,\pm}\tilde{\xi})^{2}\nonumber\\
&&\;\;\;\;\;\;\;\;\;+\mathcal{O}(|\lambda+|\tilde{\xi}||^{3})\nonumber\\
&=&\Big(-\frac{1}{a_1^{\pm}}+2\frac{b_{11}^{*,\pm}}{(a_1^{\pm})^{3}}\tilde{\xi}^{t}B_{11}^{*}\tilde{\xi}\Big)\lambda+\frac{b_{11}^{*,\pm}}{(a_1^{\pm})^{3}}\lambda^{2}-\frac{1}{a_1^{\pm}}\tilde{\xi}^{t}B_{11}^{*}\tilde{\xi}\nonumber\\
&&\;\;-\frac{b_{11}^{*,\pm}}{(a_1^{\pm})^{3}}(\tilde{\xi}\cdot\tilde{a}^{\pm})^{2}
+i(\tilde{\xi}\cdot\tilde{a}^{\pm})\Big(-\frac{1}{a^{\pm}_1}+2\frac{b^{*,\pm}_{11}\lambda}{(a_1^{\pm})^{3}}\Big)+\mathcal{O}(|\lambda+|\tilde{\xi}||^{3})
\nonumber\\
&=&(-\frac{1}{a_1^{\pm}})(\lambda+i\tilde{\xi}\cdot\tilde{a}^\pm)+\frac{b_{11}^{*,\pm}}{(a_1^{\pm})^{3}}(\lambda+i\tilde{\xi}\cdot\tilde{a}^{\pm})^{2}\nonumber\\
&&\;\;\;\;\;
-\frac{1}{a_1^{\pm}}\tilde{\xi}^{t}B_{11}^{*}\tilde{\xi}+\mathcal{O}(|\lambda+|\tilde{\xi}||^{3}).\nonumber
\end{eqnarray}
The corresponding adjoint computation yields $L^{*\pm}_1=(1,0_r)^T$.
(See, for example, the one-dimensional computation of \cite{MZ1},
which is sufficient to determine $L^{*\pm}_1$.)

\end{proof}

\textbf{Normal modes.} Consider again the variable-coefficient
eigenvalue equations
\begin{equation} \label{variablece}
(L_{\tilde{\xi}}-\lambda)W=
(dQ-A_{\tilde{\xi}}-\lambda)W-(A^{1}W)'=0
\end{equation}
and the limiting constant-coefficient equations
\begin{equation}\label{const}
(L_{\pm,\tilde{\xi}}-\lambda)W=
(dQ_{\pm}-A_{\tilde{\xi},\pm}-\lambda)W-(A^{1}_{\pm}W)'=0.
\end{equation}

We now relate the normal modes of \eqref{variablece} to those of
\eqref{const}.

\begin{lem}\textbf{(normal modes)}
Under the assumptions of Theorem \ref{main}, for $(\lambda, \tilde
\xi)\in B(0,r)$, $r$ sufficiently small ($\tilde \xi$ now complex),
there exist solutions $W_m^{\pm}(x_1;\lambda,\tilde{\xi})$ of
\eqref{variablece}, $\mathcal{C}^{2}$ in $x_1$ and analytic in
$\lambda$ and $\tilde{\xi}$, satisfying
\begin{equation}
W_m^{\pm}(x_1;\lambda,\tilde{\xi})=e^{\mu^{\pm}_m x_1}
V^{\pm}_m(\lambda,\tilde{\xi})
\end{equation}

\begin{equation}
\Big(\frac{\partial}{\partial
\lambda}\Big)^{k}\Big(\frac{\partial}{\partial
\tilde{\xi}}\Big)^{l}V^{\pm}_m(x_1;\lambda,\tilde{\xi})
=\Big(\frac{\partial}{\partial
\lambda}\Big)^{k}\Big(\frac{\partial}{\partial \tilde{\xi}}\Big)^{l}
V^{\pm}_m(\lambda,\tilde{\xi})+\mathcal{O}(e^{-\tilde{\theta}|x_1|}|V^{\pm}_m(\lambda,\tilde{\xi})|), 
\end{equation}
for any $k\geq 0$ and $0<\tilde{\theta}<\theta$, where
$\mu_m^{\pm}(\lambda,\tilde{\xi})$, $V^{\pm}_m(\lambda,\tilde{\xi})$
are as in the previous lemma.
\end{lem}

\begin{proof}
This is a direct consequence of the previous lemma and the gap
lemma, Lemma \ref{gaplemma}.
\end{proof}

\begin{prop}
There is a neighborhood of $(0,0)$ in $(\lambda,\tilde{\xi})$ space
($\tilde \xi$ now complex) in which, for $y_1\geq 0,$

\begin{equation}
G_{\lambda,\tilde{\xi}}(x_1,y_1) =
G^{1}_{\lambda,\tilde{\xi}}(x_1,y_1)+G^{2}_{\lambda,\tilde{\xi}}(x_1,y_1),
\end{equation}
where
\begin{align}\label{deco1}
G^{1}_{\lambda,\tilde{\xi}}(x_1,y_1)=
\left\{%
\begin{array}{ll}
 0& ,\;\;\; y_1\leq 0\leq x_1 \\
-\sum_j\psi^{-}_j(x_1;\lambda,\tilde{\xi})\tilde{\psi}^{-}_j(y_1;\lambda,\tilde{\xi})^{*}& ,\;\;\; y_1\leq x_1 \leq 0\\
\sum_j\varphi^{-}_j(x_1;\lambda,\tilde{\xi})\tilde{\varphi}^{-}_j(y_1;\lambda,\tilde{\xi})^{*}& ,\;\;\; x_1\leq y_1\leq 0 \\
\end{array}%
\right.
\end{align}


\begin{align}\label{deco2}
G^{2}_{\lambda,\tilde{\xi}}(x_1,y_1)=
\left\{%
\begin{array}{ll}
 \sum_{k,j}M^{+}_{jk}(\lambda,\tilde{\xi})
 \varphi^{+}_j(x_1;\lambda,\tilde{\xi})\tilde{\psi}_k^{-}(y_1;\lambda,\tilde{\xi})^{*}& ,\;\;\; y_1\leq 0\leq x_1 \\
\sum_{k,j}d^{+}_{jk}(\lambda,\tilde{\xi})\varphi^{-}_j(x_1;\lambda,\tilde{\xi})\tilde{\psi}_k^{-}(y_1;\lambda,\tilde{\xi})^{*}& ,\;\;\; y_1\leq x_1 \leq 0\\
\sum_{k,j}d^{-}_{jk}(\lambda,\tilde{\xi})\varphi^{-}_j(x_1;\lambda,\tilde{\xi})\tilde{\psi}_k^{-}(y_1;\lambda,\tilde{\xi})^{*}& ,\;\;\; x_1\leq y_1\leq 0 \\
\end{array}%
\right.
\end{align}

where
\begin{equation}
|M^{+}_{jk}|, |d^{\pm}_{jk}| \leq C_1 |D^{-1}|
\end{equation}
and
 $D(\lambda,\tilde{\xi})=\text{ det } \Phi$ is the Evans
function.

Moreover,



\begin{align}\label{GGG111}
G^{1}_{\lambda,\tilde{\xi}}(x_1,y_1)=
\left\{%
\begin{array}{ll}
 0& ,\;\;\; y_1\leq 0\leq x_1 \\
e^{\mu_{r+1}^{-}(x_1-y_1)}
\begin{pmatrix}
c_1(x_1) & 0\\
c_2(x_1)& 0\\
\end{pmatrix} \\
+\Big[e^{\mu_{r+1}^{-} (x_1-y_1)}\\
\;\;\;\;\;\;\;\;\;\; \times
\mathcal{O}\big(|\lambda|+|\tilde{\xi}|\big)\Big]
\\
\;\;\;\;\;\;\;\;\;\;\;\;+\mathcal{O}(e^{-\theta|x_1-y_1|})
\\
 & ,\;\;\; y_1\leq x_1\leq 0 \\
 \mathcal{O}(e^{-\theta|x_1-y_1|})
\\
 & ,\;\;\; x_1\leq y_1 \leq 0\\
\end{array}%
\right.
\end{align}

\begin{align}\label{yderi}
\Big(\frac{\partial}{\partial
y_1}\Big)G^{1}_{\lambda,\tilde{\xi}}(x_1,y_1)=
\left\{%
\begin{array}{ll}
 0& ,\;\;\; y_1\leq 0\leq x_1 \\
\Big[e^{\mu_{r+1}^{-} (x_1-y_1)}\\
\;\;\;\;\;\;\;\;\;\; \times
\mathcal{O}\big(|\lambda|+|\tilde{\xi}|\big)\Big]
\\
\;\;\;\;\;\;\;\;\;\;\;\;+\mathcal{O}(e^{-\theta|x_1-y_1|})
\\
 & ,\;\;\; y_1\leq x_1\leq 0 \\
 \mathcal{O}(e^{-\theta|x_1-y_1|})
\\
 & ,\;\;\; x_1\leq y_1 \leq 0\\
\end{array}%
\right.
\end{align}

\begin{align}\label{GGG222}
G^{2}_{\lambda,\tilde{\xi}}(x_1,y_1) &=
 CD^{-1}(\lambda,\tilde{\xi})e^{-\mu_{r+1}^{-}y_1}
\bar{U}'(x_1)(1,0)+\\
&\;\;\mathcal{O}\big(|D^{-1}(\lambda,\tilde{\xi})|
(|\lambda|+|\tilde{\xi}|) e^{-\frac{|x_1|}{C}} e^{ \text{ Re }
\mu^{-}_{r+1}(\lambda,\tilde{\xi})(x_1-y_1)}\big)
\nonumber\\
&\;\;\;\;+ \mathcal{O}(e^{-\theta(x_1-y_1)})\nonumber\\
\nonumber
\end{align}

\begin{align}\label{g2deri}
\Big(\frac{\partial}{\partial
y_1}\Big)G^{2}_{\lambda,\tilde{\xi}}(x_1,y_1) &=
\mathcal{O}\left(|D^{-1}(\lambda,\tilde{\xi})|
(|\lambda|+|\tilde{\xi}|) e^{-\frac{|x_1|}{C}} e^{ \text{ Re }
\mu^{-}_{r+1}(\lambda,\tilde{\xi})(x_1-y_1)}\right)
\end{align}

\end{prop}

\begin{proof}
By proposition 2.6, we decompose $G$ into two parts $G^{1}$ and
$G^{2}$ where $G^{1}$ has no poles and $G^{2}$ has all terms with
poles, respectively. Thus, we have \eqref{deco1} and \eqref{deco2}.
Without loss of generality, we can set
\begin{equation}\label{phi}\varphi_{1}=\bar{U}'(x_1)+\Lambda(x_1,\lambda,\tilde{\xi})e^{-c|x_1|}
\textup{ for } c>0
\end{equation} where $\Lambda(x_1,\lambda,\tilde{\xi})=\mathcal{O}(|\lambda|+|\tilde{\xi}|)$ is
differentiable and exponentially decaying in $x_1$. This is possible
since it solves \eqref{ev}, and is only bounded solution at zero
frequency. On the other hand, we can check by inspection that
\begin{equation}\label{slow}
\tilde{\psi}_{r+1}=[(1,0)+\Theta(y_1,\lambda,\tilde{\xi})]e^{-\mu_{r+1}|y_1|},
\end{equation}
where
$\Theta(y_1,\lambda,\tilde{\xi})=\mathcal{O}(|\lambda|+|\tilde{\xi}|)$
 is differentiable and exponentially decaying in $y_1$. By lemma 3.2,
all fast modes can be written as
\begin{align}\label{faster}
\psi^{-}_{j}&=e^{\mu_j^{-}x_1} V_j^{-}(x_1,\lambda,\tilde{\xi})\\
&=e^{\mu_j^{-}x_1}V^{-}_j(x_1,0,0)+e^{\mu_j^{-}x_1}\Psi(x_1,\lambda,\tilde{\xi})V^{-}_j(x_1,0,0)\nonumber,
\end{align}
where
$\Psi(x_1,\lambda,\tilde{\xi})=\mathcal{O}(|\lambda|+|\tilde{\xi}|)$
 is differentiable and exponentially decaying in $x_1$. Since all
modes are fast except for one slow mode
$\tilde{\psi}_{r+1}$, 
For $y_1\le x_1\le 0$, we have
\begin{align}
\sum_j\psi^{-}_j\tilde{\psi}^{-*}_j&=\psi^{-}_{r+1}\tilde{\psi}^{-*}_{r+1}+\sum_{j\ne 1}\psi^{-}_j\tilde{\psi}^{-*}_j\\
&=e^{\mu_{r+1}^{-}x_1}\big(V^{-}_j(x_1,0,0)+\Psi(x_1,\lambda,\tilde{\xi})V^{-}_j(x_1,0,0)\big)
\nonumber\\
&\;\;\;\;\;\;\;\;\; \times
e^{-\mu^{-}_{r+1}y_1}[(1,0)+\Theta(y_1,\lambda,\tilde{\xi})]+
\sum_{j\ne 1}\psi^{-}_j\tilde{\psi}^{-*}_j\nonumber\\
&= e^{\mu_{r+1}^{-}(x_1-y_1)}
\begin{pmatrix}
c_1(x_1)&0\\
c_2(x_1)&0
\end{pmatrix}\nonumber\\
& \;\;\;+ e^{\mu_{r+1}^{-} (x_1-y_1)}
\Big(\Psi(x_1,\lambda,\tilde{\xi})V^{-}_j(x_1,0,0)(1,0)\nonumber\\
&\;\;\;\;\;\;\;\;\;\;\;\;\; \;\;\;\;\;\;\;\;\;\;\;\;\;\;\;
\;\;\;\;\;\;\;\;\;\;\;\;\;\;+V^{j}(x_1,0,0)\Theta(y_1,\lambda,\tilde{\xi})\Big)\nonumber\\
&\;\;\;\;\;\;\;\;\;\;\;\; +\sum_{j\ne 1}\psi^{-}_j\tilde{\psi}^{-*}_j\nonumber\\
&=e^{\mu_{r+1}^{-}(x_1-y_1)}
\begin{pmatrix}
c_1(x_1) & 0\\
c_2(x_1)& 0\\
\end{pmatrix} \nonumber\\
&\;\;\;+\Big(e^{\mu_{r+1}^{-} (x_1-y_1)} \times
\mathcal{O}\big(|\lambda|+|\tilde{\xi}|\big)\Big)
+\mathcal{O}(e^{-\theta|x_1-y_1|}) \nonumber
\end{align}
using $\eqref{slow}$ and $\eqref{faster}$. For $x_1\leq y_1\le 0$,
since all $\varphi^{-}_j$ are fast modes, we have,
\begin{equation}
\sum_j\varphi^{-}_j(x_1;\lambda,\tilde{\xi})\tilde{\varphi}^{-}_j(y_1;\lambda,\tilde{\xi})^{*}=\mathcal{O}(e^{-\theta|x_1-y_1|}).
\end{equation}
Thus, we have $\eqref{GGG111}$.

It is easy to check $\eqref{yderi}$ by differentiating
$\eqref{GGG111}$ with respect to $y_1$.

 For $G^{2}_{\lambda,\tilde{\xi}}(x_1,y_1)$,  we
have, for $y_1\leq 0\leq x_1$,
\begin{eqnarray}
G^{2}_{\lambda,\tilde{\xi}}(x_1,y_1)&=& \sum_{j,k} M^{+}_{j,k}
\varphi^{+}_j(x_1;\lambda,\tilde{\xi})\tilde{\psi}_k^{-}(y_1;\lambda,\tilde{\xi})^{*}\nonumber\\
&=& CD^{-1} \big(e^{-\mu_{r+1}^{-}y_1} \bar{U}'(x_1)(1,0)\nonumber\\
&&\;\;\;\;\;\;\;+\Psi(x_1,\lambda,\tilde{\xi})(1,0)
e^{-cx_1}e^{-\mu_{r+1}^{-}y_1}
+\bar{U}'\Theta(y_1,\lambda,\tilde{\xi})e^{-\mu^{-}_{r+1}y_1}\nonumber\\
&&\;\;\;\;\;\;\;\;\;\;+\Psi(x_1,\lambda,\tilde{\xi})\Theta(y_1,\lambda,\tilde{\xi})e^{-cx_1}e^{-\mu^{-}_{r+1}y_1}\big)\nonumber\\
&&\;\;\;\;\;\;\;\;\;\;\;\;\; +\sum_{(j,k)\ne(1,r+1)} M^{+}_{j,k}
\varphi^{+}_j(x_1;\lambda,\tilde{\xi})\tilde{\psi}_k^{-}(y_1;\lambda,\tilde{\xi})^{*}\nonumber\\
&=&CD^{-1}(\lambda,\tilde{\xi})e^{-\mu^{-}_{r+1}y_1}\bar{U}'(x_1)(1,0)
\nonumber\\
&&\;\;\;\;+\mathcal{O}(|\lambda|+|\tilde{\xi}|)e^{-\frac{|x_1|}{C}}
e^{ \text{ Re } \mu^{-}_{r+1}(\lambda,\tilde{\xi})(x_1-y_1)}
D^{-1}(\lambda,\tilde{\xi})
\nonumber\\
&&\;\;\;\;\;\;+ \mathcal{O}(e^{-\theta(x_1-y_1)})\nonumber.
\end{eqnarray}
Similarly, for $x_1\leq y_1\leq 0$ and for $y_1\leq x_1\leq 0$,
respectively, we have the same bounds. It is also easy to check
$\eqref{g2deri}$ by differentiating $\eqref{GGG222}$ with respect to
$y_1$.

%
%
%
%
%
%
%
%
\end{proof}

\begin{cor}
There is a neighborhood of $(0,0)$ in $(\lambda,\tilde{\xi})$ space
($\tilde \xi $ now complex) in which, for $y_1\geq 0,$
\begin{align}\label{g1}
G^{1}_{\lambda, \tilde{\xi}}(x_1,y_1)
\begin{pmatrix}
0\\
I\\
\end{pmatrix}=
\left\{%
\begin{array}{ll}
 0& ,\;\;\; y_1\leq 0\leq x_1 \\
\Big[e^{\textup{ Re }\mu_{r+1}^{-} (x_1-y_1)}\\
\;\;\;\;\;\;\;\;\;\; \times
\mathcal{O}\big(|\lambda|+|\tilde{\xi}|\big)\Big]
\\
\;\;\;\;\;\;\;\;\;\;\;\;+\mathcal{O}(e^{-\theta|x_1-y_1|})
\\
 & ,\;\;\; y_1\leq x_1\leq 0 \\
 \mathcal{O}(e^{-\theta|x_1-y_1|})
\\
 & ,\;\;\; x_1\leq y_1 \leq 0\\
\end{array}%
\right.
\end{align}

\begin{align}\label{g2}
\displaystyle G^{2}_{\lambda, \tilde{\xi}}(x_1,y_1)
(0,I)^{t}&= \mathcal{O}\big(|D^{-1}(\lambda,\tilde{\xi})|
(|\lambda|+|\tilde{\xi}|) e^{-\frac{|x_1|}{C}} e^{ \text{ Re }
\mu^{-}_{r+1}(\lambda,\tilde{\xi})(x_1-y_1)}\big)
\end{align}
\end{cor}

\begin{proof}
If we multiply \eqref{GGG111} by $(0,I)^{t}$, then we have the
result.
\end{proof}

\begin{rem}\label{derivrmk}
\textup{ Note that $G^{j}_{\lambda, \tilde{\xi}}(x_1,y_1) (0,I)^{t}$
has almost the same bounds as
$\frac{\partial}{\partial_{y_1}}G^{j}_{\lambda,
\tilde{\xi}}(x_1,y_1)$ for $j=1,2$. This will be important in our
later derivation of pointwise Green function bounds, yielding that
differentiation by $y$ is roughly comparable to right-multiplication
by $(0,I)^{t}$. }
\end{rem}

\subsection{Decomposition of the Green function}$\ $
 For fixed small $\delta_1$, $\theta_1$, $r>0$ to be chosen later, define a
``low-frequency" part $G^{I}$ and a ``high-frequency" part $G^{II}$
of $G$, respectively by
\begin{eqnarray}
&&G^{I}(x,t;y):=\nonumber\\
&&\;\;\;\;\;\;\;\;\;\;\;\; 
\frac{1}{(2\pi i)^{d}}
\int_{|\tilde{\xi}|\leq \delta_1}
\oint_{\Gamma_0}
e^{i\tilde{\xi}\cdot (\tilde{x}-\tilde{y}) +\lambda t} G_{\lambda,\tilde{\xi}}(x_1,y_1)  d\lambda d\tilde{\xi},\nonumber\\
\nonumber
\end{eqnarray}
where  $\Gamma_0=[-\eta-ir,\eta-ir]\cup[\eta-ir,\eta+ir] \cup
[2\eta+ir,-\eta-ir]$, and
\begin{eqnarray}
&&G^{II}(x,t;y):=\nonumber\\
&&\;\;\;\;\;\;\;\;\;\;\;\;\;
\frac{1}{(2\pi i)^{d}}
\textup{P.V.}\int_{-\theta_1-i\infty}^{-\theta_1+i\infty}
\int_{|\tilde \xi|\ge \delta_1 \, \hbox{\rm or }\, |\Im \lambda|\ge
r} e^{i\tilde{\xi}\cdot (\tilde{x}-\tilde{y}) +\lambda t}
G_{\lambda,\tilde{\xi}}(x_1,y_1) d\tilde{\xi} d\lambda.
\nonumber\\
\nonumber \end{eqnarray}

Then, by the spectral resolution formula \eqref{spectral} together
with Cauchy's Theorem, we have a decomposition formula for
$G(x,t;y)$ of
\begin{equation}
G(x,t;y)=G^{I}(x,t;y)+G^{II}(x,t;y).
\end{equation}

\section{Low-frequency estimates}\label{LF}
We estimate the low frequency part of the Green function following
\cite{HoZ2}.

\subsection{Pointwise bounds}
The stationary phase approximation
$$
U\sim \bar U'(x_1)\delta(\tilde x,t)
$$
of \eqref{sp} can be expressed alternatively (recalling that the
low-frequency part of $G$ generally determines large-time behavior)
as
\begin{equation}
G^{I}(x,t;y)\sim \bar U'(x_1)([u]^{-1},0)g(\tilde x - \tilde y,t),
\label{5.20a}
\end{equation}
where $g(\tilde x,t)$ denotes the Green function for the
constant-coefficient $(d-1)$-dimensional equation \eqref{conv-diff}
approximately governing normal deformation of the front. Using the
analysis of \cite{HoZ1} together with the low-frequency description
of the resolvent kernel in Proposition \ref{Gdecomp}, we now
establish the following pointwise description of the low-frequency
part of the Green function $G^{I}$, sharpening the formal prediction
of \eqref{5.20a}.

\begin{prop}\label{Gdecomp}
 Under the assumptions of Theorem \ref{main},
for $y_1\ge 0$, $|\alpha|\le 1$, and some $\eta$, $M>0$,
\begin{align}\label{g1description}
D_y^\alpha G^{I}&= \chi_{\{|x_1-y_1| \le |a_1^+t|\}} \chi_{\{0\ge
y_1+a_1^+t\}} D_y^\alpha \bar U'(x_1)([u]^{-1},0)
\bar g^+(\tilde x,t; y)\\
&\hspace{10mm} + \chi_{ \{x_1\ge 0\}} D_y^\alpha K^+(x,t;y)
+R^+_\alpha,\nonumber
\end{align}
where
\begin{equation}\label{bargreens}
\displaystyle \bar g^+(\tilde x, t;y):= c_{\tilde{\beta}}t^{-\frac
{d-1}{2}} e^{ -\frac{ (\tilde x - \tilde y - { \tilde
\alpha}^+(y_1,t)t)^t {\bar \beta}^{-1}_+ (\tilde x - \tilde y -
{\tilde \alpha}^+(y_1,t)t)} {4t} },
\end{equation}
and
\begin{equation}\label{K}
K^+(x,t;y):= c_{\mathcal{B}^{*}_+} t^{-\frac{d}{2}} e^{
-\frac{(x-y-a^+t)^t \mathcal{B}_+^{*-1}(x-y-a^+t) } {4t}}
\end{equation}
are $(d-1)$- and $d$-dimensional convected heat kernels,
respectively, with
\begin{equation}
{ \tilde \alpha}^+:= (1- \frac{|y_1|}{|a_1^+t|}) \bar {\tilde a} +
\frac{|y_1|}{|a_1^+t|}\tilde a^+, \label{alpha}
\end{equation}
\begin{equation}
\bar \beta_+:= (1- \frac{|y_1|}{|a_1^+t|}) \tilde{\beta}
+\frac{|y_1|}{|a_1^+t|} \big( b^{*}_{11, +}\bar B_+^{*} +
\frac{b^{*}_{11, +}}{|a_1^+|^2} (\tilde a^+-\bar {\tilde a} -
a_1^+b^{* t}) (\tilde a^+-\bar {\tilde a} - a_1^+b^{* t})^t \big),
\label{barbeta}
\end{equation}
\begin{equation}
\bar B^{*}:= B^{*}- b^{*}b^{* t}, \label{barB}
\end{equation}
where $\tilde a^+$ is given by $a^+=:(a_1^+,\tilde
a^+)=(df_{1}^{*}(u_+),df_{2}^{*}(u_+),...df_{d}^{*}(u_+))$, $b^{*}$,
$B^{*}$, $\mathcal{B}^{*}$, $\bar {\tilde a}$, and $\tilde{\beta}$
are as in \eqref{BBstar} and \eqref{abars},
and the (faster decaying)
residual terms $R_\alpha^+$
and $\Theta_\alpha^+$ satisfy
\begin{align}\label{R}
 |R^+_\alpha|&\le
\chi_{\{|x_1- y_1| \le |a_1^+t|\}} t^{-\frac{d-1}{2} -
\frac{|\alpha|}{2}}((1+t)^{-\frac{1}{2}}+ \alpha_1 e^{-\eta|y_1|})
%
e^{-\eta|x_1|}e^{-\frac{|\tilde x -\tilde y -\tilde a t|^2}{Mt}} \\
&\hspace{10mm}+ C e^{-\frac{|x -y -a^+ t|^2}{Mt}}
t^{-\frac{d-1}{2}-\frac{|\alpha|}{2}}e^{-\eta|x_1|}\nonumber\\
&\hspace{10mm}+ C e^{-\frac{|x -y -a^+ t|^2}{Mt}}
t^{-\frac{d}{2}-\frac{|\alpha|}{2}}((1+t)^{-\frac{1}{2}}
\chi_{\{x_1\ge 0\}}+ e^{-\eta t})\nonumber\\
&\hspace{10mm} +C\chi_{\{|\tilde x- \tilde y| \ge M(t+|x_1-y_1|)\}}
\frac{ e^{-\eta_1(|x_1-y_1|+t)}}{\Pi_{j=2}^d (1+|x_j-y_j|)}.
  \nonumber
\end{align}
A symmetric description holds for $y_1\le 0$.
\end{prop}

Before proving Proposition \ref{Gdecomp}, we first establish the
following lemma, a simplified version of Proposition 2.8 in
\cite{HoZ2}, allowing us to vary $\tilde \xi$ in the complex plane.
Following \cite{HoZ2}, denote $\tilde \xi\in \CC$ by
\begin{equation}\label{xicomplex}
\tilde \xi=\xi_1+i\xi_2.
\end{equation}

\begin{lem}\label{move}
Under the assumptions of Theorem \ref{main}, for some $\delta_2$,
$\eta_1>0$ sufficiently small, $|\tilde x-\tilde y|\le
M(t+|x_1-y_1|)$, and any $|\tilde \xi_2^*|\le \delta_2$,
\begin{equation}\label{cGI}
\begin{aligned}
G^{I}(x,t;y)&= 
\frac{1}{(2\pi i)^{d}}
\int_{|\tilde{\xi_1}|\leq \delta_1, \, \tilde
\xi_2=\tilde \xi_2^*} \oint_{\Gamma_0} e^{i\tilde{\xi}\cdot
(\tilde{x}-\tilde{y}) +\lambda t}
G_{\lambda,\tilde{\xi}}(x_1,y_1)  d\lambda d\tilde{\xi}\\
&\quad + O(e^{-\eta_1(|x-y|+t)}).
\end{aligned}
\end{equation}
\end{lem}

\begin{proof}
By Cauchy's Theorem, it is equivalent to show
\begin{equation}\label{equiv}
\int_{|\tilde{\xi_1}|=\delta_1, \, |\tilde \xi_2|\le \delta_2}
\oint_{\Gamma_0} e^{i\tilde{\xi}\cdot (\tilde{x}-\tilde{y}) +\lambda
t} G_{\lambda,\tilde{\xi}}(x_1,y_1)  d\lambda d\tilde{\xi}=
 O(e^{-\eta_1(|x-y|+t)}).
\end{equation}
By assumption ($\mathcal{D}1$) and continuity, there are no zeros of
the Evans function for $|\tilde \xi_1|=\delta_1$ and $\Re \lambda
\ge -\theta$, some $\theta>0$. Thus, by Cauchy's Theorem again,
\eqref{equiv} is equivalent in turn to
\begin{equation}\label{equiv2}
\int_{|\tilde{\xi_1}|=\delta_1, \, |\tilde \xi_2|\le \delta_2}
\int_{-\eta-ir}^{-\eta + ir} e^{i\tilde{\xi}\cdot
(\tilde{x}-\tilde{y}) +\lambda t} G_{\lambda,\tilde{\xi}}(x_1,y_1)
d\lambda d\tilde{\xi}=
 O(e^{-\eta_1(|x-y|+t)}).
\end{equation}

Taking $\eta<<\delta_1$, we have on the domain of integration by our
previous bounds (see for example Remark \ref{genbd}) that
$$
|G_{\lambda,\tilde{\xi}}(x_1,y_1)|\le Ce^{-\theta_2|x_1-y_1|}
$$
for $\theta_2>0$. Taking $\delta_2<<\theta_2$, $\eta$, we thus have
$$
|e^{i\tilde{\xi}\cdot (\tilde{x}-\tilde{y}) +\lambda t}
G_{\lambda,\tilde{\xi}}(x_1,y_1)|\le Ce^{-\theta_2|x_1-y_1|} \le
Ce^{-\eta t/2}e^{-\theta_2|x_1-y_1|/2},
$$
yielding the result for $\eta_1:=(1/2)\min\{\theta_2, \eta\}$.
\end{proof}

We have also the following weakened version of Proposition 2.7,
\cite{HoZ2}.

\begin{lem}\label{small}
Under the assumptions of Theorem \ref{main}, for $M$ sufficiently
large and $|\tilde x-\tilde y|\ge M(t+|x_1-y_1|)$,
\begin{equation}\label{smallGI}
|G^{I}(x,t;y)|\le C\frac{ e^{-\eta_1(|x_1-y_1|+t)}}{ \Pi_{j=2}^d
(1+|x_j-y_j|)}.
\end{equation}
\end{lem}

\begin{proof}
We first consider the simplest case of dimension $d=2$. Moving
$\tilde \xi_2^*$ from $0$ to $c:=\delta_2 (\tilde x-\tilde
y)/|\tilde x-\tilde y|$ in
$$
\int_{|\tilde{\xi_1}|\leq \delta_1, \, \tilde \xi_2=\tilde \xi_2^*}
\oint_{\Gamma_0} e^{i\tilde{\xi}\cdot (\tilde{x}-\tilde{y}) +\lambda
t} G_{\lambda,\tilde{\xi}}(x_1,y_1)  d\lambda d\tilde{\xi},
$$
$\delta_2>0$ fixed, by the argument of Lemma \ref{move} yields a
change of order
$$
e^{-\eta(t+|x_1-y_1|)}\int_0^{\delta_2}e^{-z|\tilde x-\tilde y|}dz
\le \frac{Ce^{-\eta(t+|x_1-y_1})}{1+|\tilde x-\tilde y|}.
$$
On the other hand,
$$
\begin{aligned}
\int_{|\tilde{\xi_1}|\leq \delta_1, \, \tilde \xi_2=c}
\oint_{\Gamma_0} e^{i\tilde{\xi}\cdot (\tilde{x}-\tilde{y}) +\lambda
t} G_{\lambda,\tilde{\xi}}(x_1,y_1)  d\lambda d\tilde{\xi}
&\le Ce^{-\delta_2 |\tilde x-\tilde y|} e^{C_1(t+ |x_1-y_1|)}\\
&\le Ce^{-( \delta_2/2)(t+ |x-y|)},
\end{aligned}
$$
since by assumption $|\tilde x-\tilde y|>>(t+|x_1-y_1|)$. This
completes the proof for dimension $d=2$.

For dimensions $d>2$, we proceed by induction, moving one component
of $\tilde \xi_2$ at a time, starting with the component for which
$|x_j-y_j|$ is largest, without loss of generality, $j=2$.  At the
first step, then, this yields a change of order
$$
e^{-\eta(t+|x_1-y_1|)}\int_0^{\delta_2}e^{-z|x_2-y_2|}dz \le
\frac{Ce^{-\eta(t+|x_1-y_1})}{1+|x_2-y_2|}
$$
times the maximum of a family of $(d-2)$-dimensional integrals in
$(\xi_3,\dots, \xi_d)$ of similar form to the original one, plus a
new integral of negligible order $O(e^{-( \delta_2/2)(t+ |x-y|)})$.
Moving the next component similarly, in each of the family of
$(d-2)$-dimensional integrals, yields a factor
$$
\frac{C}{1+|x_3-y_3|}
$$
times the maximum of a family of $(d-3)$-dimensional integrals of
similar form, plus a new integral of smaller order $O(e^{-\eta
|x_3-y_3|})$ times another $(d-3)$-dimensional integral of similar
form.  Continuing this process, we obtain the result.
\end{proof}

\begin{rem}\label{inefficient}
\textup{ For $|\tilde x-\tilde y|\ge Mt$, $M$ sufficiently large, we
have $G\equiv 0$ by finite speed of propagation. Thus,
\eqref{smallGI} reflects a certain inefficiency in our splitting
scheme. Note that the righthand side is time-exponentially decaying
in $L^p$, $p>1$, whereas usual error terms $O(e^{-\eta(t+ |x-y|)})$
are time-exponentially decaying in all $L^p$ norms; thus, for
practical purposes it is almost but not quite optimal. For the
present analysis, in which we consider only $2\le p\le \infty$, it
is harmless. }
\end{rem}

\begin{proof}[Proof of Proposition \ref{Gdecomp}]
With our preparations, this follows now by exactly the argument used
in \cite{HoZ2} to establish the corresponding bounds on the full
solution operator for $|x-y|\ge Mt$, $M>>1$. 
(Note: what is actually
estimated for this regime in \cite{HoZ2} is the low-frequency part,
with the rest shown to be negligible.)

Specifically, we note that the description of the resolvent kernel
in Proposition \ref{G1} agrees with that for the viscous case in
Proposition 2.5, \cite{HoZ2} in the sense that the principal terms
are identical, with the rest consisting of fast-decaying
($O(e^{-\theta|x_1-y_1|})$) terms leading to negligible errors.  The
only difference in the relaxation case is that there are more (at
most $r$) of the latter terms than in the viscous case (at most
$1$).

Then, the rest of the proof goes word for word as in the arguments
of \cite{HoZ2}, Sections 3, 4, and 5, based on careful stationary
phase estimates of the various terms. For this (complicated)
argument, we refer to \cite{HoZ2}.
\end{proof}

\subsection{$L^1\to L^p$ bounds}
From Proposition \ref{Gdecomp}, we readily obtain the following
bounds on the solution operator.

\begin{lem}\label{Lqp}
For $t\geq1$ and $\alpha$ is a multi-index with $|\alpha|\leq1$,
there holds
\begin{equation}\label{pwbds}
\big|D_x^{\alpha}G^{I}(x,t;y)\big|_{L^{p}(x)}\leq C(p)
t^{((d-1)/2)(1-1/p)-|\alpha|/2}
\end{equation}
for all $p>1$, with $C\to \infty$ as $p\to 1$.
 Moreover, for $f\in L^{1}$ and $p>1$, there holds
\begin{equation}\label{g11}
\Big|\int G^{I}(x,t;y)f(y) dy\Big|_{L^{p}(x)} \leq
C(p)(1+t)^{-((d-1)/2)(1-1/p)}|f|_{L^{1}}
\end{equation}
and
\begin{equation}\label{g22}
\Big|\int G^{I}_{y_j}(x,t;y)f(y) dy\Big|_{L^{p}(x)} \leq
C(p)(1+t)^{-((d-1)/2)(1-1/p)-1/2}|f|_{L^{1}}.
\end{equation}
\end{lem}

\begin{proof}
By calculating the $L^{p}$ norm of the $(d-1)$-dimensional heat
kernel, together with the expression of $\eqref{g1description}$, we
have
\begin{align}
|D^\alpha_x G^{I}|_{L^{p}(x)}&\leq C\big(|\bar U '(x_1)
D_x^{\alpha} \bar{g}^+|_{L^{p}}+ |D^\alpha_x \bar K^+|_{L^{p}}+|R^+_\alpha|_{L^{p}} \big)\\
&\leq C t^{((d-1)/2)(1-1/p)-|\alpha|/2} +C(p)e^{-\eta t},
\nonumber\\
\nonumber
\end{align}
$\eta>0$. The inequalities $\eqref{g11}$ and $\eqref{g22}$ follows
from $\eqref{pwbds}$ and the triangle inequality.
%
\end{proof}

\begin{lem}
For $p>1$ and $C(p)$ as in Lemma \ref{Lqp},
\begin{equation}\label{fastpart}
\Big| \int G^{I}(x,t,y)(0,I_r)^{t}f(y)dy\Big|_{L^{p}}\leq
C(p)(1+t)^{-(d-1)/2((1-1/p)-1/2)}|f|_{L^{1}}.
\end{equation}
\end{lem}

\begin{proof}
By $\eqref{g1}$ and $\eqref{g2}$, we know that $G^{2}_{\lambda,
\tilde{\xi}}(x_1,y_1) (0,I)^{t}$ has the same bound as
$\frac{\partial}{\partial_{y_1}}G^{j}_{\lambda,
\tilde{\xi}}(x_1,y_1)$. Thus, $\Big| \int
G^{I}(x,t,y)(0,I_r)^{t}f(y)dy\Big|_{L^{p}}$ has the same bound as
$\Big|\int G^{I}_{y_j}(x,t;y)f(y) dy\Big|_{L^{p}(x)}$. By
\eqref{g22}, we have the result.
\end{proof}

\begin{lem}
The low frequency part of Green function $G^{I}(x,t;y)$ associated
with $(\partial_t-L)$ satisfies
\begin{align}\label{Gpig}
&\Big| \int \big(G^{I}(x,t;y)-\bar{U}'(x_1)\Pi^{t}(x_1) g(\tilde{x}-\tilde{y},t)\big)f(y) dy\Big|_{L^{p}(x)}\\
&\leq C t^{-((d-1)/2)(1-1/p)-1/2}\Big(\big|x_1f(x)\big|_{L^{1}(x)} +
\big|f(x)\big|_{L^{1}(x)}\Big)\nonumber
\end{align}
for all $t\geq 0$, where $\Pi(x_1)$ is the left zero eigenvector
dual to the right zero eigenvector $\bar{U}'(x_1)$ at
$\tilde{\xi}=0$
\end{lem}

\begin{proof}
By \eqref{g1description}, we have
\begin{align}\label{g1ubarpi}
&|G^{I}(x,t;0,\tilde{y})-\bar{U}'(x_1)\Pi^{t}(x_1)
g(\tilde{x}-\tilde{y},t)\big)f(y) dy|_{L^{p}(x)}\\
&\leq C t^{-((d-1)/2)(1-1/p)-1/2}.\nonumber
\end{align}
Then using \eqref{g1ubarpi}, we have
\begin{align}\label{ggg}
&\Big|\int G^{I}(x,t;0,\tilde{y})-\bar{U}'(x_1)\Pi^{t}(x_1) g(\tilde{x}-\tilde{y},t)\big)f(y) dy\Big|_{L^{p}(x)}\\
& \leq |\int G^{I}(x,t;0,\tilde{y})-\bar{U}'(x_1)\Pi^{t}(x_1) g(\tilde{x}-\tilde{y},t)\big)f(y) dy|_{L^{p}(x)}|f|_{L^{1}(x)}\nonumber\\
& \leq C t^{-((d-1)/2)(1-1/p)-1/2}.\nonumber
\end{align}
On the other hand, we have
\begin{align}\label{difference}
&\Big|\int G^{I}(x,t;y_1,\tilde{y})-G^{I}(x,t;0,\tilde{y})\big)f(y) dy\Big|_{L^{p}(x)}\\
&\;\;\;\;\;\;\;\;=\Big| \int \int_0^{1} \partial_{y_1} G^{I}(x,t;\theta y_1,\tilde{y}) y_1 f(y) d\theta dy\Big|_{L^{p}(x)}\nonumber\\
&\;\;\;\;\;\;\;\;\leq \sup_y |\partial_{y_1}G^{I}(x,t;y)|_{L^{p}(x)}|x_1f(x)|_{L^{p}(x)}\nonumber\\
&\;\;\;\;\;\;\;\;\leq t^{-((d-1)/2)(1-1/p)-1/2}|x_1f(x)|_{L^{p}(x)}.\nonumber\\
\nonumber
\end{align}
Combining \eqref{ggg} and \eqref{difference} together with the
triangle inequality, we have the result.
\end{proof}

\section{Damping and high frequency estimates}\label{sec:damping}
We now carry out the main new work of the paper, establishing
high-derivative and high-frequency bounds by energy estimates
following the approach introduced in \cite{Z4} for the
hyperbolic-parabolic case.

We denote $U=(u,v)^{t}$ for our convenience. Let $\tilde{U}$ be a
solution of \eqref{main} and $\bar{U}$ be a traveling wave solution
of \eqref{main}. Define the nonlinear perturbation
$U(x,t):=\tilde{U}(x,t)-\bar{\bar{U}}(x,t)$ where
$\bar{\bar{U}}(x,t)=\bar{U}(x_1-\delta(\tilde{x},t))$. We also
denote $\tilde{A}=A(\tilde{U})$, $\bar{A}=A(\bar{U})$,
$d\tilde{Q}=dQ(\tilde{U})$ and $d\bar{Q}=dQ(\bar{U})$.
\begin{lem}
Multi-dimensional scalar relaxation equations \eqref{main} can be
put in a quasilinear form as follows: \label{quasi1}
\begin{align}
& U_t+\sum_j\tilde{A}^{j}U_{x_j}-d\tilde{Q} U\\
&=(\partial_t-L)\delta(\tilde{x},t)\bar{U}'(x_1)-R_{x_1}+(0,I_r)^{t}N(U,U)-M=:f\nonumber
\end{align}
where
\begin{equation}
M=\sum_{j=1}^{d}(\tilde{A}^{j}-\bar{\bar{A}}^{j})\bar{\bar{U}}_{x_j}=O(|U||\bar{U}'||\nabla_{\tilde{x}}\delta
|),
\end{equation}
\begin{equation}
R_{x_1}=O(|\delta_t||\bar{U}'||\delta|+ |\nabla_{\tilde{x}}\delta |
|\bar{U}'|
|\delta|)_{x_1}=O(|\delta|(|\delta_t|+|\nabla_{\tilde{x}}\delta|)|\bar{U}''|)
\end{equation}
and \begin{equation} N(U,U)=O(|U|^{2}).
\end{equation}
\end{lem}
\begin{proof}

 Consider the multi-dimensional scalar relaxation equations.

 \begin{equation}
\begin{pmatrix}
u\\
v
\end{pmatrix}_t
+\sum_{j=1}^{d}
\begin{pmatrix}
f^{j}(u,v)\\
g^{j}(u,v)
\end{pmatrix}_{x_j}
=
\begin{pmatrix}
0\\q(u,v)
\end{pmatrix}
\end{equation}
For $\bar{\bar{U}}$, there holds
\begin{align}\label{doublebar}
&\bar{\bar{U}}_t + \sum_j \begin{pmatrix}
f^{j}(\bar{\bar{U}})\\
g^{j}(\bar{\bar{U}})
\end{pmatrix}_{x_j}-\begin{pmatrix}
o\\q(\bar{\bar{U}})
\end{pmatrix}{}\\
&=-\delta_t \bar{U}'\big|_{x_1-\delta}+\sum_{j\neq
1}(-\delta_{x_j})\begin{pmatrix} d f^{j}(\bar{U})\\d g^{j}(\bar{U})
\end{pmatrix}\bar{U}'\Big|_{x_1-\delta}+\underbrace{\begin{pmatrix}
f^{1}(\bar{\bar{U}})\\
g^{1}(\bar{\bar{U}})
\end{pmatrix}_{x_1}-\begin{pmatrix}
0\\q(\bar{\bar{U}})\end{pmatrix}}_{=0}\nonumber\\
&=-\delta_t \bar{U}'(x_1)+\sum_{j\neq
1}(-\delta_{x_j})\begin{pmatrix} d f^{j}(\bar{U}(x_1))\\d
g^{j}(\bar{U}(x_1))
\end{pmatrix}\bar{U}'(x_1)+\underbrace{\begin{pmatrix}
f^{1}(\bar{U})\\
g^{1}(\bar{U})
\end{pmatrix}_{x_1}-\begin{pmatrix}
0\\q(\bar{U})\end{pmatrix}}_{=0} -R_{x_1} \nonumber\\
&=-(\partial_t-L)\delta(\tilde{x},t) \bar{U}'(x_1)-R_{x_1},\nonumber
\end{align}

 where
 \begin{equation}
 R_{x_1}=-\delta_t
\bar{U}'\big|_{x_1-\delta}^{x_1}+\sum_{j\neq
1}(-\delta_{x_j})\begin{pmatrix} d f^{j}(\bar{U})\\d g^{j}(\bar{U})
\end{pmatrix}\bar{U}'\Big|_{x_1-\delta}^{x_1}.
\end{equation}

The last equality is true since the stationary shock wave solution
$\bar{U}'(x_1)$ satisfies

\begin{align}
0&=\begin{pmatrix}
\begin{pmatrix}
f^{1}(\bar{U})\\
g^{1}(\bar{U})
\end{pmatrix}_{x_1}-\begin{pmatrix}
0\\q(\bar{U})\end{pmatrix}
\end{pmatrix}_{x_1} \nonumber\\
&= \begin{pmatrix}
\begin{pmatrix}
df^{1}(\bar{U}(x_1))\\dg^{1}(\bar{U}(x_1))
\end{pmatrix}\bar{U}'(x_1)
\end{pmatrix}_{x_1}-\begin{pmatrix}
0\\dq(\bar{U}(x_1))
\end{pmatrix} \bar{U}'(x_1).
\end{align}

The last equation is obtained by multiplying the shock wave
equaition by $\delta$ and adding to the second last equation.

Let $\tilde{U}$ be a solution of
\begin{equation}\label{main1}
\tilde{U}_t +\sum_j
\begin{pmatrix}
f^{j}(\tilde{U})\\
g^{j}(\tilde{U})
\end{pmatrix}_{x_j} = \begin{pmatrix}
0\\
q(\tilde{U})
\end{pmatrix}.
\end{equation}
After subtracting \eqref{doublebar} from \eqref{main1}, we have the
perturbation equation for $U$


\begin{align}
&U_t+\sum_{j=1}^{d}\big(A^{j}(\tilde{U})\tilde{U}_{x_j}-A^{j}(\bar{\bar{U}})\bar{\bar{U}}_{x_j}\big)
-\big(Q(\tilde{U})-Q(\bar{\bar{U}})\big)\\
&=U_t+\sum_{j=1}^{d}\big(F^{j}(\tilde{U})-F^{j}(\bar{\bar{U}})\big)_{x_j}-
\big(Q(\tilde{U})-Q(\bar{\bar{U}})\big)\nonumber\\
&=(\partial_t-L)\delta\bar{U}'-R_{x_1}.\nonumber
\end{align}
Keeping a quasilinear form on the left hand side, we have
\begin{align}\label{quasi}
&U_t+\sum_j\tilde{A}^{j}U_{x_j}-d\tilde{Q}U\\
&=(\partial_t-L)\delta(\tilde{x},t)\bar{U}'(x_1)-R_{x_1}+(0,I_r)^{t}N(U,U)-\sum_{j=1}^{d}(\tilde{A}^{j}-\bar{\bar{A}}^{j})\bar{\bar{U}}_{x_j}\nonumber\\
&=(\partial_t-L)\delta(\tilde{x},t)\bar{U}'(x_1)-R_{x_1}+(0,I_r)^{t}N(U,U)-M=:f\nonumber
\end{align}
where
\begin{equation}
M=\sum_{j=1}^{d}(\tilde{A}^{j}-\bar{\bar{A}}^{j})\bar{\bar{U}}_{x_j}=O(|U||\bar{U}'||\nabla_{\tilde{x}}\delta
|),
\end{equation}
\begin{equation}
R_{x_1}=O(|\delta_t||\bar{U}'||\delta|+ |\nabla_{\tilde{x}}\delta |
|\bar{U}'|
|\delta|)_{x_1}=O(|\delta|(|\delta_t|+|\nabla_{\tilde{x}}\delta|)|\bar{U}''|)
\end{equation} and
\begin{equation}
N(U,U)=O(|U|^{2})
\end{equation} since
$\tilde{A}^{j}-\bar{\bar{A}}^{j}=A^{j}(\tilde{U})-A^{j}(\bar{\bar{U}})=\int_0^{1}
dA^{j}\big(\bar{\bar{U}}+\theta(\tilde{U}-\bar{\bar{U}})\big)U
d\theta=O(U)$ and $\bar{\bar{U}}_{x_j}=-\bar{U}'\delta_{x_j}$.
\end{proof}

\begin{ass}
\text{The operator $K(\partial_x)$ is defined by }
\begin{equation}
\widehat{K(\partial_x)f}(\xi)=i \bar{K}(\xi) \widehat{f}(\xi)
\end{equation}
\text{where $\bar{K}(\xi)$ is a skew-symmetric operator which are smooth}\\
\text{and homogeneous degree one in $\xi$ satisfying}

\begin{equation}\label{kawashimacondition}
\Re \sigma\big(|\xi|^{2}A^{0}dQ-\sum_{j=1}^{d}\xi_j
\bar{K}(\xi)A^{j}\big)_{\pm} \leq -\theta |\xi|^{2} \textup{   for all }\xi \text{ in } \mathbb{R}^{d}\\
\end{equation}
\end{ass}

\begin{rem}
\textup{
This is a standard assumption of Kawashima which is satisfied when
\eqref{abb} is simultaneously symmetrizable and satisfies
genuine-coupling condition. (See \cite{Yo} and \cite{Ze}.)
}
\end{rem}

\begin{prop}[Damping estimate]\label{4.4}
If $ |U|_{H^{s}}(t)\leq \varepsilon$ sufficiently small for $0\leq
t\leq T$ where $s\geq[\frac{d}{2}]+2$, then, for some
$\tilde{\theta}>0$, there holds
\begin{equation}\label{pro}
|U|_{H^{s}}^{2}(t)\leq e^{-\theta t}|U|_{H^{s}}^{2}(0)
+C\int_0^{t}e^{-\theta(t-s)}\big(|U|_{L^{2}}^{2}(s)+\epsilon(s)\big)ds
\end{equation} for $0\leq t\leq T$, where
\begin{equation}
\epsilon(t)\leq C |\nabla_{t,\tilde{x}}\delta|_{L^{2}}^{2}=
C\zeta_0^{2} (1+t)^{-\frac{d-1}{2}-1}.
\end{equation}
\end{prop}

\begin{proof}
Let $\alpha$ be a multi-index with $|\alpha|=r\geq 1$.
Taking a differential operator $\partial_x^{\alpha}$ on the equation
\eqref{quasi}, we have
\begin{equation}
\partial_{x}^{\alpha} U_{t}+\sum_{j=1}^{d}\partial_{x}^{\alpha}\big(\tilde{A}^{j}U_{x_j}\big)-\partial_{x}^{\alpha}\big(d\tilde{Q}U\big)=\partial_{x}^{\alpha} f
\end{equation}
where
$f:=(\partial_t-L)\delta(\tilde{x},t)\bar{U}'(x_1)-R_{x_1}+(0,I_r)^{t}N(U,U)-M.$

Taking the $L^{2}$ inner product of $A^{0}\partial_x^{\alpha}U$
against $\partial_{x}^{\alpha} U$, we have the energy estimate:
\begin{align}\label{mainestimate}
\frac{1}{2}\frac{d}{dt}\big<A^{0}\partial_x^{\alpha}U,\partial_{x}^{\alpha}
U\big>&=\frac{1}{2}\big<A^{0}_t\partial_x^{\alpha}U,\partial_{x}^{\alpha}
U\big>+\big<A^{0}\partial_x^{\alpha}U_t,\partial_{x}^{\alpha}U\big>\\
&\leq \frac{1}{2}\big<(A^{0}_t+\sum_j
(A^{0}\tilde{A}^{j})_{x_j})\partial_x^{\alpha}U,\partial_{x}^{\alpha}U\big>
+\big<(A^{0}d\tilde{Q})\partial_x^{\alpha}U,\partial_x^{\alpha}U\big>\nonumber\\
& \;\;\;\;\;-C\big<\sum_{k=1}^{|\alpha|}\sum_j
A^{0}(\partial^{k}_x\tilde{A}^{j})
(\partial^{|\alpha|-k}_{x}U_{x_j}),\partial^{\alpha}_x U\big>\nonumber\\
& \;\;\;\;\;+\sum_{k=1}^{|\alpha|}\big<A^{0}(\partial^{k}_x
d\tilde{Q})(\partial^{|\alpha|-k}_{x}U),
\partial_x^{\alpha}U\big>+\big<A^{0}\partial_x^{\alpha}f,\partial_x^{\alpha}U\big>\nonumber\\
& \leq
C\big(|U|_{W^{1,\infty}}+|U|_{W^{1,\infty}}^{r}\big)|U|_{H^{r}}^{2}+
\varepsilon|U|^{2}_{H^{r}}+C|U|_{L^{2}}^{2}\nonumber\\
& \;\;\;\;\;\;\;\;
+\big<A^{0}d\tilde{Q}\partial_x^{\alpha}U,\partial_x^{\alpha}U\big>
+\epsilon(t),\nonumber\\
& \leq
C\big(|U|_{W^{1,\infty}}+|U|_{W^{1,\infty}}^{r}\big)|U|_{H^{r}}^{2}+
\varepsilon|U|^{2}_{H^{r}}+C|U|_{L^{2}}^{2}\nonumber\\
& \;\;\;\;\;\;\;\; +\big<(A^{0}dQ)_-
\partial_x^{\alpha}U,\partial_x^{\alpha}U\big>
+\epsilon(t),\nonumber
\end{align}where
\begin{equation}\nonumber\epsilon(t)=O\big(|\nabla_{t,\tilde{x}}\delta|_{L^{2}}^{2}\big)=
C(1+t)^{-\frac{d-1}{2}-1}.
\end{equation} The second last inequality is true
by Moser's inequalities and Sobolev inequalities.
For each $\alpha$ with $|\alpha|=r\geq 1$, we define
$\tilde{\alpha}:=\alpha-e_j$ where \text{ $j =$ min }$\{\;i\; : \;
\alpha_i \text{ is maximal}\}$. Then, $|\tilde{\alpha}|=r-1$. Let
$S_r:=\{(\alpha,\tilde{\alpha}):\;\;|\alpha|=r\}$. Similarly, taking
the $L^{2}$ inner product of
$K(\partial^{\alpha-\tilde{\alpha}}_x)\partial_x^{\tilde{\alpha}}U$
against $\partial_x^{\tilde{\alpha}}U$, we have the auxiliary energy
estimate:
\begin{align}\label{skewKK}
 \frac{1}{2}\frac{d}{dt}\langle
K(\partial^{\alpha-\tilde{\alpha}}_x)\partial_x^{\tilde{\alpha}}U,\partial_x^{\tilde{\alpha}}U\rangle&=
\frac{1}{2}\frac{d}{dt}\langle i
\bar{K}(\xi^{\alpha-\tilde{\alpha}})(i\xi)^{\tilde{\alpha}}\hat{U},(i\xi)^{\tilde{\alpha}}\hat{U}\rangle
\\
&=\langle i \bar{K}(\xi^{\alpha-\tilde{\alpha}})(i\xi)^{\tilde{\alpha}}\hat{U},(i\xi)^{\tilde{\alpha}}\hat{U}_t\rangle\nonumber\\
&=\langle (i\xi)^{2\tilde{\alpha}}\hat{U},-\sum_j\bar{K}(\xi^{\alpha-\tilde{\alpha}})\xi_jA_-^{j}\hat{U}\rangle\nonumber\\
&\;\;\;\;\; + \langle i
\bar{K}(\xi^{\alpha-\tilde{\alpha}})(i\xi)^{\tilde{\alpha}}\hat{U},(i\xi)^{\tilde{\alpha}}\hat{H}\rangle.\nonumber
\end{align}
using Plancherel's inequality together with the equation
\begin{equation}
\hat{U}_t=-\sum_j i\xi_j A_-^{j}\hat{U}+\hat{H},
\end{equation}
where
\begin{equation}
H:=\sum_j( A_-^{j}-\tilde{A}^{j}) U+(d \tilde{Q})U+f.
\end{equation}
By a direct calculation with the Moser inequality and the assumption
of smallness of $\bar{U}$, we have
\begin{equation}
|(i\xi)^{\tilde{\alpha}}\hat{H}|_{L^{2}(\xi)}=|\partial_x^{\tilde{\alpha}}
H|_{L^{2}(x)}\leq
C|U|_{L^{\infty}}|U|_{H^{r-1}}+\varepsilon|U|_{H^{r-1}}.
\end{equation}
Thus, we have
\begin{align}\label{kaw1}
&\frac{1}{2}\frac{d}{dt}\sum_{|\alpha|=r}\Big(
\big<A^{0}\partial_x^{\alpha}U,\partial_{x}^{\alpha}
U\big>+\big<K(\partial^{\alpha-\tilde{\alpha}}_x)\partial^{\tilde{\alpha}}_x
U,\partial_x^{\tilde{\alpha}}U\big>\Big)\\ \nonumber
&\leq \sum_{(\alpha,\tilde{\alpha})\in S_r}\big<
\xi^{2\tilde{\alpha}}\big(\xi^{2(\alpha-\tilde{\alpha})}(A^{0}dQ)_-
-\sum_{j=1}^{d}\xi_j\bar{K}(\xi^{\alpha-\tilde{\alpha}})A^{j}_-\big)
\hat{U},\hat{U}\big>\nonumber\\
&\;\;\;\;\;\;\;\;\;\;\;+\varepsilon|U|^{2}_{H^{r}}+C|U|_{L^{2}}^{2}+\epsilon(t)\nonumber\\
&\leq -\theta
\sum_{|\alpha|=r}\big<\xi^{2\alpha}\hat{U},\hat{U}\big>+\varepsilon|U|^{2}_{H^{r}}+C|U|_{L^{2}}^{2}+\epsilon(t)
\nonumber\\
 &\leq-\theta
|U|_{H^{r}}^{2}+\varepsilon|U|^{2}_{H^{r}}+C|U|_{L^{2}}^{2}+\epsilon(t)\nonumber
\end{align}
The second last inequality is true by \eqref{kawashimacondition}. By
\eqref{kaw1} and choosing $\varepsilon\leq\theta/2$, we obtain
\begin{align}\label{finalest}
&\frac{1}{2}\frac{d}{dt}\Big(\big<A^{0}U,U\big>+
\sum_{r=1}^{s}\sum_{|\alpha|=r}c^{-r}\big(\big<A^{0}\partial_x^{\alpha}U,\partial_{x}^{\alpha}
U\big>+\big<K(\partial^{\alpha-\tilde{\alpha}}_x)
\partial^{\tilde{\alpha}}_x
U,\partial_x^{\tilde{\alpha}}U\big>\big)\Big)\\
&\leq \sum_{r=1}^{s}\sum_{(\alpha,\tilde{\alpha})\in S_r}c^{-r}\big<
\xi^{2\tilde{\alpha}}\big(\xi^{2(\alpha-\tilde{\alpha})}(A^{0}Q)_-
-\sum_{j=1}^{d}\xi_j\bar{K}(\xi^{\alpha-\tilde{\alpha}})A_-^{j}\big)
\hat{U},\hat{U}\big> \nonumber\\
&\;\;\;\;\;\;\;\;\;\;\;\;\;+\sum_{r=1}^{s} \frac{\theta}{2c^{r}} |U|_{H^{r}}^{2}+C|U|_{L^{2}}^{2}+\epsilon(t)\nonumber\\
&\leq
-\frac{\theta}{2c^{s}}|U|_{H^{s}}^{2}+C|U|_{L^{2}}^{2}+\epsilon(t)\nonumber
\end{align}
so long as $|U|_{W^{1.\infty}}$ is small. We define
\begin{align}\label{equi}
\mathcal{E}(t)&:=\big<A^{0}U, U\big>+
\sum_{r=1}^{s}\sum_{|\alpha|=r}c^{-r}\big(\big<A^{0}\partial_x^{\alpha}U,\partial_{x}^{\alpha}
U\big>+\big<K(\partial^{\alpha-\tilde{\alpha}}_x)
\partial^{\tilde{\alpha}}_x
U,\partial_x^{\tilde{\alpha}}U\big>\big)\\
&\sim |U|_{H^{s}}^{2}(t)\nonumber
\end{align}
%
%
It is easy to check that $\mathcal{E}(t)$ is equivalent to
$|U|_{H^{s}}^{2}(t)$. Using \eqref{finalest} and \eqref{equi}, we
have the Gronwall-type inequality
\begin{equation}
\frac{d}{dt}\mathcal{E}(t)\leq -\tilde{\theta}
\mathcal{E}(t)+C\big(|U|_{L^{2}}^{2}(t)+\epsilon(t)\big).
\end{equation}
Therefore, we have the result.
\end{proof}

\begin{lem}\textup{ (High-Frequency Operator Estimate) }\label{HFest}
Let $G^{II}$ be the high-frequency part of Green function associated
with $(\partial_t-L)$. For any $f \in H^{3}$, there holds
\begin{equation}\label{hfestimate}
\Big| \int G^{II}(x,t;y)f(y)dy\Big|_{L^{2}}\leq C e^{-\theta
t}|f|_{H^{3}}.
\end{equation}
for some $\theta>0$.
\end{lem}

\begin{proof}
First, we establish the ``high-frequency'' resolvent bound by
Kawashima-type energy estimate. To this end, consider the eigenvalue
equation
\begin{equation}
(\lambda-L_{\tilde{\xi}})U=f
\end{equation}
Set
\begin{align}
&\big\langle A^{0}(\lambda-L_{\tilde{\xi}})(1+|\tilde{\xi}|+\partial_{x_1})U, (1+|\tilde{\xi}|+\partial_{x_1})U\big\rangle\\
&\;\;\;\;\;\;\;\;\;\;\;=\big\langle
A^{0}(1+|\tilde{\xi}|+\partial_{x_1})f,(1+|\tilde{\xi}|+\partial_{x_1})U\big\rangle.\nonumber
\end{align}
We obtain by the simpler linear version of Kawashima-type energy
estimate as in Proposition \ref{4.4} (the damping estimate), that
\begin{align}\label{realpart}
\text{ Re } \lambda \big|(1+|\tilde{\xi}|+\partial_{x_1})U\big|^{2}_{L^{2}(x_1)} &\leq -\theta \big|(1+|\tilde{\xi}|+\partial_{x_1})U\big|^{2}_{L^{2}(x_1)} \\
&\;+C^{*}\big|(1+|\tilde{\xi}|+\partial_{x_1})f\big|^{2}_{L^{2}(x_1)}+C^{*}\big|U\big|^{2}_{L^{2}(x_1)}.
\nonumber
\end{align}

Taking the imaginary part of the $L^{2}$ inner product of $U$
against $\lambda U=L_{\tilde{\xi}}U +f$, we have
\begin{align}\label{impart}
&|\text{Im }\lambda| |U|^{2}_{L^{2}(x_1)}\leq \\
& \;\;\;C_2 |f|^{2}_{L^{2}(x_1)}
+(2\varepsilon)^{-1}|U|^{2}_{L^{2}(x_1)}+\varepsilon|(1+|\tilde{\xi}|+\partial_{x_1})U|^{2}_{L^{2}(x_1)}.\nonumber
\end{align}

By \eqref{realpart} and \eqref{impart}, we establish the
high-frequency resolvent bound
\begin{equation}\label{resolventid}
\big|(\lambda-L_{\tilde{\xi}})^{-1}\big|_{H^{1}(x_1)}\leq C \textup{
for } |(\tilde{\xi},\lambda)|\geq R \textup{ and
}\mathfrak{R}\lambda\geq -\theta,
\end{equation}
for some $R, C > 0$ sufficiently large and $\theta >0$ sufficiently small.\\
Moreover, we have the intermediate frequency bound,
\begin{equation}
|(\lambda-L_{\tilde{\xi}})^{-1}|_{H^{1}(x_1)}\leq C \textup{ for }
R^{-1}\leq |(\tilde{\xi},\lambda)|\leq R \textup{ and
}\mathfrak{R}\lambda\geq -\theta,
\end{equation}
for any $R>0$ and $C=C(R)>0$ sufficiently large and $\theta>0$
sufficiently small. This follows by compactness of the set of
frequencies under consideration together with the fact that the
resolvent $(\lambda-L_{\tilde{\xi}})^{-1}$ is analytic with
respect to $H^{1}$ in $(\tilde{\xi}, \lambda)$.

On the other hand, it is easy to check the following resolvent
identity using analyticity on the resolvent set $\rho(L)$ of the
resolvent $(\lambda-L_{\tilde{\xi}})^{-1}$, for all $f\in
\mathcal{D}(L_{\tilde{\xi}})$,
\begin{equation}
(\lambda-L_{\tilde{\xi}})^{-1}f=\lambda^{-2}(\lambda-L_{\tilde{\xi}})^{-1}L_{\tilde{\xi}}^{2}f+\lambda^{-2}L_{\tilde{\xi}}f+\lambda^{-1}f
\end{equation}
Using the resolvent identity $\eqref{resolventid}$, the
high-frequency solution operator $S^{II}$ can be written as
\begin{eqnarray}
S^{II}f&=& \textup{P.V.}\int_{-\theta_1-i\infty}^{-\theta_1+i\infty}
\int_{\mathbb{R}^{d-1}}\chi_{|\tilde{\xi}|^{2}+|\mathfrak{I}\lambda|^{2}\geq\theta_1+\theta_2}\nonumber\\
&&\;\;\;\;\;\;\;\;\;\;\;\;\;\;\;\;\; \times e^{i\tilde{\xi}\cdot (\tilde{x}-\tilde{y}) +\lambda t} (\lambda-L_{\tilde{\xi}})^{-1} \hat{f}(x_1,\tilde{\xi}) d\tilde{\xi} d\lambda \nonumber\\
&=& \textup{P.V.} \int_{-\theta_1-i\infty}^{-\theta_1+i\infty}
\int_{\mathbb{R}^{d-1}}\chi_{|\tilde{\xi}|^{2}+|\mathfrak{I}\lambda|^{2}\geq\theta_1+\theta_2}\nonumber\\
&&\;\;\;\;\;\;\;\;\;\;\;\;\;\;\;\;\; \times e^{i\tilde{\xi}\cdot
(\tilde{x}-\tilde{y}) +\lambda
t}\lambda^{-2}(\lambda-L_{\tilde{\xi}})^{-1}L_{\tilde{\xi}}^{2}\hat{f}(x_1,\tilde{\xi}) d\tilde{\xi} d\lambda\nonumber\\
&& + \textup{P.V.}\int_{-\theta_1-i\infty}^{-\theta_1+i\infty}
\int_{\mathbb{R}^{d-1}}\chi_{|\tilde{\xi}|^{2}+|\mathfrak{I}\lambda|^{2}\geq\theta_1+\theta_2}\nonumber\\
&&\;\;\;\;\;\;\;\;\;\;\;\;\;\;\;\;\; \times e^{i\tilde{\xi}\cdot
(\tilde{x}-\tilde{y}) +\lambda t}\lambda^{-2}L_{\tilde{\xi}}\hat{f}(x_1,\tilde{\xi}) d\tilde{\xi} d\lambda \nonumber\\
&& + \textup{P.V.}\int_{-\theta_1-i\infty}^{-\theta_1+i\infty}
\int_{\mathbb{R}^{d-1}}\chi_{|\tilde{\xi}|^{2}+|\mathfrak{I}\lambda|^{2}\geq\theta_1+\theta_2}\nonumber\\
&&\;\;\;\;\;\;\;\;\;\;\;\;\;\;\;\;\; \times e^{i\tilde{\xi}\cdot
(\tilde{x}-\tilde{y}) +\lambda t}\lambda^{-1}\hat{f}(x_1,\tilde{\xi}) d\tilde{\xi} d\lambda \nonumber\\
&=:& A + B +C.\nonumber
\end{eqnarray}
For $f\in H^{3}$, there holds
\begin{align}\label{A}
|A|_{H^{1}}&\leq C e^{-\theta_1 t} \textup{ sup }_{\lambda} | (\lambda-L_{\tilde{\xi}})^{-1}|_{H^{1}} |f|_{H^{3}} \\
&\leq C^{*} e^{-\theta_1 t} |f|_{H^{3}}.\nonumber
\end{align}
Similarly, we have
\begin{align}\label{B}
|B|_{H^{1}}&\leq C e^{-\theta_1 t} |f|_{H^{2}}.
\end{align}
Using the triangle inequality, we have
\begin{align}\label{C}
|C|_{H^{1}}&\leq\Big|\text{P.V.}\int_{-\theta_1-i\infty}^{-\theta_1+i\infty}\lambda^{-1}e^{\lambda t} d\lambda \int_{\RR^{d-1}}e^{i\tilde{x}\cdot\tilde{\xi}}\hat{f}(x_1,\tilde{\xi}) d\tilde{\xi}\Big|_{H^{1}}\\
&\;\;\;\;\; + \Big|\int_{-\theta_1-i r}^{-\theta_1+i r}\lambda^{-1}e^{\lambda t} d\lambda \int_{\RR^{d-1}}e^{i\tilde{x}\cdot\tilde{\xi}}\hat{f}(x_1,\tilde{\xi}) d\tilde{\xi}\Big|_{H^{1}}\nonumber\\
&\leq 2 r / \theta_1 e^{-\theta_1 t} |f|_{H^{1}}.\nonumber
\end{align}
The last inequality is true since the first term is zero,
by the inverse Laplace transform identity
$$
\text{P.V.}\frac{1}{2\pi i}\int_{-\theta_1-i\infty}^{-\theta_1+i\infty}
\lambda^{-1}e^{\lambda t} d\lambda 
=\begin{cases} 0 &\theta_1 t>0\\
1 &\theta_1 t<0.\\
\end{cases}
$$

By the expression of $S^{II}$, together with $\eqref{A},$
$\eqref{B}$ and $\eqref{C}$, we have
\begin{equation}
|S^{II}|_{H^{3}\rightarrow L^{2}}(t)\leq C e^{-\theta t}. 
\end{equation}
Therefore, we have the result.
\end{proof}

\begin{rem}
\textup{ The argument of Lemma \ref{HFest}, based on the
initial-value problem with homogeneous forcing, greatly simplifies
the treatment in \cite{Z4} based on the initial-value problem with
homogeneous data and inhomogeneous forcing. The argument is quite
general, in particular extending without modification to the case
that the equilibrium model is a system. }
\end{rem}

\section{Nonlinear $L^{2}$ decay and asymptotic behavior.}
We now carry out the nonlinear analysis by an argument combining the
approach of \cite{HoZ2} in the low-frequency domain with that of
\cite{Z4} in the high-frequency domain.
 Let the approximate shock deformation $\delta(\tilde{x},t)$ be a solution of the constant coefficient equation
 \begin{equation}
 \delta_t+\bar{\tilde{a}}\cdot\nabla_{\tilde{x}}\delta=\text{div}_{\tilde{x}}(\tilde{\beta}\nabla_{\tilde{x}}\delta)
 \end{equation}
 with initial data
 \begin{equation}\label{convdiffconst}
\delta_0(\tilde{x})=-([u]^{-1},0) \int_{-\infty}^{\infty}
\tilde{U}(x_1,\tilde{x},0)-\bar{U}(x_1) dx_1.
\end{equation}
 Here, $\bar{\tilde{a}}$ and $\tilde{\beta}$ are as in Appendix \ref{appA}.

 We define a smooth approximation of $\delta$, denoted by
$\delta^{\varepsilon}$, by
 \begin{equation} \delta^{\varepsilon}(\cdot,t):=\eta^{\varepsilon} * \delta(\cdot,t)
 \end{equation}
 where $\eta^{\varepsilon}$ is a smooth mollifier supported in $B(0,\varepsilon)$. Note that $\delta^{\varepsilon}$
satisfies the same convected heat equation as $\delta$ does with
$\mathcal{C}^{\infty}$ initial data $\delta^{\varepsilon}_0
=\eta^{\varepsilon}*\delta_0$. Define
\begin{equation}\label{res}
U(x,t):=\tilde{U}(x,t)-\bar{U}(x_1-\delta^{\varepsilon}(\tilde{x},t)).
\end{equation}

 \begin{lem}\label{deltalem}
 For $|\tilde{U}_0-\bar{U}|_{L^{1}}\leq \zeta_0$,  a multi-index $|\alpha|\leq K$, there holds
 \begin{equation}\label{deltabound}
 |\partial_{\tilde{x}}^{\alpha}\delta^{\varepsilon}(\cdot,t)|_{L^{2}}\leq C\zeta_0(1+t)^{-(d-1)/4-|\alpha|/2}
 \end{equation}
 and
 \begin{equation}
 |\delta^{\varepsilon}(\cdot,t)-\delta(\cdot,t)|_{L^{2}}\leq C\zeta_0 t^{-(d-1)/4-1/2}
 \end{equation} where $C=C(\varepsilon, K)$ is a constant.
 \end{lem}
 \begin{proof}
 Letting $g(\tilde{x},t)$ be a Green function for \eqref{convdiffconst}, we have
 \begin{equation}
 \delta^{\varepsilon}(\cdot,t)=g(\cdot,t)*\delta^{\varepsilon}(\cdot,0)
 \end{equation} and
 \begin{equation}
 \delta^{\varepsilon}(\cdot,t)-\delta(\cdot,t)=g(\cdot,t)*(\delta^{\varepsilon}(\cdot,0)-\delta(\cdot,0)).
\end{equation}

Thus, we have
\begin{align}\label{longtime}
 |\partial_{\tilde{x}}^{\alpha}\delta^{\varepsilon}(\cdot,t)|_{L^{2}}
 &\leq |\partial_{\tilde{x}}^{\alpha}g(\cdot,t)|_{L^{2}}|\delta^{\varepsilon}(\cdot,t)|_{L^{1}}\\
 &\leq |\partial_{\tilde{x}}^{\alpha}g(\cdot,t)|_{L^{2}}|\eta^{\varepsilon}|_{L^{1}}|\delta^{\varepsilon}_0|_{L^{1}}\nonumber\\
 &\leq C\zeta_0 t^{-(d-1)/4-|\alpha|/2}\nonumber
\end{align}
by the standard fact that a heat kernel $g$ decays as
\begin{equation}
|\partial_{\tilde{x}}^{\alpha}g(\cdot,t)|_{L^{2}}\leq C
t^{-(d-1)/4-|\alpha|/2}
\end{equation}
and observing that $|\delta_0|_{L^{1}}\leq
C|\tilde{U}_0-\bar{U}|_{L^{1}}\leq C\zeta_0$ by definition of
$\delta_0$. On the other hand, we have also
\begin{align}\label{shorttime}
|\partial_{\tilde{x}}^{\alpha}\delta^{\varepsilon}(\cdot,t)|_{L^{2}}&\leq |g(\cdot,t)|_{L^{1}}|\partial_{\tilde{x}}^{\alpha}\eta^{\varepsilon}|_{L^{2}}|\delta(\cdot,t)|_{L^{1}}\\
&\leq C\zeta_0\varepsilon^{-|\alpha|}.\nonumber
\end{align}
By \eqref{longtime} and \eqref{shorttime}, we have the first claim.
Expressing
\begin{equation}
\delta^{\varepsilon}(\cdot,t)-\delta(\cdot,t)=(g^{\varepsilon}(\cdot,t)-g(\cdot,t))*\delta(\cdot,0),
\end{equation}
and noting that
\begin{align}
|(g^{\varepsilon}(\tilde{y},t)-g(\tilde{y},t))|_{L^{2}(\tilde{y})}&=
\Big|\int( g(\tilde{y}-\tilde{z},t)-g(\tilde{y},t))\eta^{\varepsilon}(\tilde{z})d\tilde{z}\Big|_{L^{2}(\tilde{y})}\\
&=\Big|\int \int_0^{1} (\nabla_{\tilde{x}}g(\tilde{y}-\theta \tilde{z},t)\cdot \tilde{z}) \eta^{\varepsilon}(\tilde{z})d\theta d\tilde{z}\Big|_{L^{2}(\tilde{y})}\nonumber\\
&\leq \int_0^{1} |\nabla_{\tilde{x}}g(\tilde{y}-\theta \tilde{z},t)|_{L^{2}(\tilde{y})}|\tilde{z}\eta^{\varepsilon}(\tilde{z})|_{L^{1}(\tilde{z})} d\theta\nonumber\\
&\leq C \varepsilon t^{-(d-1)/4-1/2}.\nonumber
\end{align}
In the last inequality, we have used the fact that
$|\tilde{z}\eta^{\varepsilon}(\tilde{z})|_{L^{1}(\tilde{z})}\leq
C\varepsilon$. Thus, we obtain
\begin{align}
|\delta^{\varepsilon}(\cdot,t)-\delta(\cdot,t)|_{L^{2}}&\leq |g^{\varepsilon}(\cdot,t)-g(\cdot,t)|_{L^{2}}|\delta(\cdot,0)|_{L^{1}}\\
&\leq C\varepsilon\zeta_0 t^{-(d-1)/4-1/2}.\nonumber
\end{align}
\end{proof}

\begin{lem}
If $| \tilde{U}_0-\bar{U}|_{L^{1}}$, $|
\tilde{U}_0-\bar{U}|_{L^{\infty}}$, $|
x_1(\tilde{U}_0-\bar{U})|_{L^{1}} \leq \zeta_0$. Then,
\begin{equation}
|U_0|_{L^{1}}, |U_0|_{L^{\infty}}, | x_1 U_0|_{L^{1}}\leq
C(\varepsilon)\zeta_0.
\end{equation}
\end{lem}

\begin{proof}
We have
\begin{equation}
\bar{U}(x_1)-\bar{U}(x_1-\delta_0^{\varepsilon})=\int_0^{1}
\bar{U}'(x_1-\theta\delta_0^{\varepsilon})\delta_0^{\varepsilon}d\theta.
\end{equation}
Then we have
\begin{align}\label{differencewithdelta}
&|\bar{U}(x_1)-\bar{U}(x_1-\delta_0^{\varepsilon})|_{L^{p}(x)}\\
&\leq \int_0^{1}
|\bar{U}'(x_1-\theta\delta_0^{\varepsilon})\delta_0^{\varepsilon}|_{L^{p}(x)}d\theta
\nonumber\\
&\leq \int_0^{1}
|\bar{U}'|_{L^{p}(x_1)}|\delta_0^{\varepsilon}|_{L^{p}(\tilde{x})}d\theta\leq
C\zeta_0\nonumber
\end{align}
By \eqref{differencewithdelta} together with the triangle
inequality, we have
\begin{equation}
|U_0|_{L^{p}(x)}\leq|\tilde{U}_0-\bar{U}|_{L^{p}(x)}+
|\bar{U}(x_1)-\bar{U}(x_1-\delta_0^{\varepsilon})|_{L^{p}(x)}\leq(C+1)\zeta_0.
\end{equation}
Similarly, we can prove the other inequality.
\end{proof}

\begin{lem} The nonlinear residual $U(x,t)$ defined in $\eqref{res}$ satisfies
\begin{align}
&U_t+\sum_{j=1}^{d}\big(A^{j}(\bar{U})U\big)_{x_j}- d Q
(\bar{U})U\\
&=(\partial_t-L)\delta^{\varepsilon}\bar{U}'-R_{x_1}+\sum_{j=1}^{d}N^{j}_{x_j}+
(0,I_r)^{t}N^{0}+\sum_{j=1}^{d} S^{j}_{x_j} +
(0,I_r)^{t}S^{0}\nonumber
\end{align}
where
\begin{equation}\label{5.20}
R_{x_1}=O(|\delta^{\varepsilon}_t||\bar{U}'||\delta^{\varepsilon}|+
|\nabla_{\tilde{x}}\delta^{\varepsilon} | |\bar{U}'|
|\delta^{\varepsilon}|)_{x_1}=O(|\delta^{\varepsilon}|(|\delta^{\varepsilon}_t|+|\nabla_{\tilde{x}}\delta^{\varepsilon}|)|\bar{U}''|)
\end{equation} and
\begin{equation}\label{5.21}
N^{j}=O(|U|^{2}) \text{ and } S^{j}=O(|\delta^{\varepsilon}|
|\bar{U}'| |U|)\text{ for }j=0,1,...,d.
\end{equation}
\end{lem}
\begin{proof}
Let $\tilde{U}$ be a solution of
\begin{align}\label{utilde}
&
\tilde{U}_t+\sum_{j=1}^{d}A^{j}(\tilde{U})\tilde{U}_{x_j}-Q(\tilde{U})=0.
\end{align}

For
$\bar{\bar{U}}(x,t)=\bar{U}(x_1-\delta^{\varepsilon}(\tilde{x},t))$,
there holds
\begin{align}\label{ubarbar}
&\bar{\bar{U}}_t
+\sum_{j=1}^{d}A^{j}(\bar{\bar{U}})\bar{\bar{U}}_{x_j}-Q(\bar{\bar{U}})\\
&=-\delta^{\varepsilon}_t\bar{U}'+\sum_{j=2}^{d}A^{j}(\bar{\bar{U}})(-\delta^{\varepsilon}_{x_j})\bar{U}'\nonumber\\
&=-\delta^{\varepsilon}_t\bar{U}'(x_1)+\sum_{j=2}^{d}A^{j}(\bar{U}(x_1))(-\delta^{\varepsilon}_{x_j})\bar{U}'(x_1)+R_{x_1}
\nonumber\\
&=-(\partial_t-L)\delta^{\varepsilon}\bar{U}'+R_{x_1}.\nonumber
\end{align}
If we subtract \eqref{ubarbar} from \eqref{utilde}, we have
\begin{align}\label{subtractfromutilde}
&U_t+\sum_{j=1}^{d}\big(A^{j}(\tilde{U})\tilde{U}_{x_j}-A^{j}(\bar{\bar{U}})\bar{\bar{U}}_{x_j}\big)
-\big(Q(\tilde{U})-Q(\bar{\bar{U}})\big)\\
&=U_t+\sum_{j=1}^{d}\big(F^{j}(\tilde{U})-F^{j}(\bar{\bar{U}})\big)_{x_j}-
\big(Q(\tilde{U})-Q(\bar{\bar{U}})\big)\nonumber\\
&=(\partial_t-L)\delta\bar{U}'-R_{x_1}\nonumber
\end{align}
By Taylor expansion of $F^{j}$ about $\bar{\bar{U}}$ and
\eqref{subtractfromutilde}, we have
\begin{align}\label{subtractfromutilde2}
&U_t+\sum_{j=1}^{d}\big(A^{j}(\bar{\bar{U}})U\big)_{x_j}- d Q
(\bar{\bar{U}})U\\
&=(\partial_t-L)\delta\bar{U}'-R_{x_1}+\sum_{j=1}^{d}N^{j}_{x_j}+
(0,I_r)^{t}N^{0}(U,U)
\nonumber
\end{align}
By Taylor expansion of $\bar{\bar{A}}^{j}$ about $\bar{U}$ and
\eqref{subtractfromutilde2}, we have
\begin{align}
&U_t+\sum_{j=1}^{d}\big(A^{j}(\bar{U})U\big)_{x_j}- d Q
(\bar{U})U\\
&=(\partial_t-L)\delta^{\varepsilon}\bar{U}'-R_{x_1}+\sum_{j=1}^{d}N^{j}_{x_j}+
(0,I_r)^{t}N^{0}
\nonumber\\
&\;\;\;\;\;\;\;\;+\sum_{j=1}^{d}\Big(\big(A^{j}(\bar{U})-A^{j}(\bar{\bar{U}})\big)U\Big)_{x_j}
-\big(d Q(\bar{U})-d Q(\bar{\bar{U}})\big) U \nonumber\\
&=(\partial_t-L)\delta^{\varepsilon}\bar{U}'-R_{x_1}+\sum_{j=1}^{d}N^{j}_{x_j}+
(0,I_r)^{t}N^{0}+\sum_{j=1}^{d} S^{j}_{x_j} +
(0,I^{r})S^{0}\nonumber
\end{align}
where
\begin{equation}
R_{x_1}=O(|\delta^{\varepsilon}_t||\bar{U}'||\delta^{\varepsilon}|+
|\nabla_{\tilde{x}}\delta^{\varepsilon} | |\bar{U}'|
|\delta^{\varepsilon}|)_{x_1}=O(|\delta^{\varepsilon}|(|\delta^{\varepsilon}_t|+|\nabla_{\tilde{x}}\delta^{\varepsilon}|)|\bar{U}''|),
\end{equation}
\begin{equation}
N^{j}=O(|U|^{2})\text{ for }j=0,1,...,d
\end{equation} and
\begin{equation}
S^{j}=O(|\delta^{\varepsilon}| |\bar{U}'| |U|)\text{ for
}j=0,1,...,d.
\end{equation}

\end{proof}

\begin{lem}
For $f \in \mathcal{C}^{2} \cap L^{1}$, there holds
\begin{align}\label{keycancel}
&\int_0^{t} \int G(x,t-s;y) (\partial_s-L_y) f(y,s) dyds\\
&\hspace{10mm}=f(x,t)-\int G(x,t;y)f(y,0)dy\nonumber
\end{align}
\end{lem}
\begin{proof} Integrating by parts, we have
\begin{align}
&\int_0^{t-\varepsilon} \int G(x,t-s;y) (\partial_s-L_y) f(y,s) dyds\nonumber\\
&=\int G(x,\varepsilon;y)f(y,t-\varepsilon)dy-\int G(x,t;y)f(y,0) dy\nonumber\\
&\;\;\;\;\;+ \int_0^{t-\varepsilon} \int
(\partial_s-L_y)^{*}G(x,t-s;y) f(y,s) dyds.
\end{align}
By duality and letting $\varepsilon\rightarrow 0$, we obtain the
result.

\end{proof}

We are now ready to prove our main result.

\begin{proof}[Proof of Theorem \ref{mainthm}]
By Lemma \ref{deltalem}, it is sufficient to show
\begin{equation}\label{resultmaineps}
\big|\tilde{U}(x,t)-\bar{U}(x_1-\delta^{\varepsilon}(\tilde{x},t))\big|_{L^{2}(x)}\leq
C\zeta_0(1+t)^{-(d-1)/4-1/2+\sigma}.
\end{equation}
Define
\begin{equation}\label{zeta}
\zeta(t):=\text{sup }_{0\leq s\leq t} |U(\cdot,s)|_{L^{2}}(1+s)^{(d-1)/4+1/2-\sigma}.
\end{equation}
\begin{claim}For all $t\geq 0$, there holds
\begin{equation}
\zeta(t) \leq C_1(\zeta^{2}(t)+\zeta_0\zeta(t)+\zeta_0^{2})\leq
C_2(\zeta_0+\zeta^{2}(t)).
\end{equation}
\end{claim}
From this result, it follows by continuous induction that
\begin{equation}
\zeta(t)\leq 2 C_2\zeta_0 \text{ for all } t\geq 0,
\end{equation}
provided $\zeta_0<\frac{1}{4C_2}$, i.e., $\zeta(t)$ remains small
for all $t\geq 0$. \smallbreak \textbf{Proof of claim.} Applying
Duhamel's principle, we can express
\begin{align}\label{duhamel}
U(x,t)&=\Big(\int G(x,t;y)U_0(y)dy - \int_0^{t}\int G(x,t-s;y)(\partial_s-L)\delta^{\varepsilon}\bar{U}'(y,s)dyds\Big)\\
&\;\;\;\;\;+\int_0^{t}\int G^{I}(x,t-s;y)(-R_{y_1}+\sum_{j=1}^{d}N^{j}_{y_j}+\sum_{j=1}^{d} S^{j}_{y_j})(y,s)dyds\nonumber\\
&\;\;\;\;\;+\int_0^{t}\int G^{II}(x,t-s;y)(-R_{y_1}+\sum_{j=1}^{d}N^{j}_{y_j}+\sum_{j=1}^{d} S^{j}_{y_j})(y,s)dyds\nonumber\\
&\;\;\;\;\;+\int_0^{t}\int G^{I}(x,t-s;y)\big((0,I_r)^{t}N^{0}+ (0,I_r)^{t}S^{0}\big)(y,s)dyds\nonumber\\
&\;\;\;\;\;+\int_0^{t}\int G^{II}(x,t-s;y)\big((0,I_r)^{t}N^{0}+ (0,I_r)^{t}S^{0}\big)(y,s)dyds\nonumber\\
&=I+II+III+IV+V\nonumber.
\end{align}

First, we establish that
\begin{equation}
|I|_{L^{2}}\leq C\zeta_0(1+t)^{-(d-1)/4-1/2}.
\end{equation}
Using \eqref{Gpig}, we have
\begin{align}\label{g1u0}
&\Big|\int G^{I}(x,t;y)U_0(y)dy\Big|_{L^{2}}\\
&\leq \Big|\bar{U}'(x_1)\int g(\tilde{x}-\tilde{y},t)\big(\Pi^{t}\int U(y_1,\tilde{y},0)dy_1\big)d\tilde{y}\Big|_{L^{2}}(x)\nonumber\\
&\;\;\;\;\;+ C\zeta_0 t^{-(d-1)/4-1/2}\nonumber\\
&=\Big|\bar{U}'(x_1)(\delta^{\varepsilon}(\tilde{x},t)-\delta(\tilde{x},t))\Big|_{L^{2}(x)}+C\zeta_0 t^{-(d-1)/4-1/2}\nonumber\\
&\leq C_*\zeta_0 t^{-(d-1)/4-1/2}\nonumber.
\end{align}
By standard $C^0$ semigroup theory, we obtain the short-time bound
\begin{equation}\label{stb}
\Big|\int G(x,t;y)U_0(y)dy\Big|_{L^{2}(x)}\leq C|U_0|_{L^{2}}\leq
C\zeta_0 \text{ for } t\leq 1.
\end{equation}
By \eqref{g1u0} and \eqref{stb}, we have
\begin{equation}
\Big|\int G^{I}(x,t;y)U_0(y)dy\Big|_{L^{2}(x)}\leq
C\zeta_0(1+t)^{-(d-1)/4-1/2}.
\end{equation}
By \eqref{hfestimate},
\begin{equation}
\Big|\int G^{II}(x,t;y)U_0(y)dy\Big|_{L^{2}} \leq e^{-\theta t}
|U_0|_{H^{3}}\leq C\zeta_0(1+t)^{-(d-1)/4-1/2}.
\end{equation}

On the other hand, by \eqref{keycancel}, we have
\begin{align}
&\Big|\int_0^{t}\int G(x,t-s;y)(\partial_s-L)\delta^{\varepsilon}\bar{U}'(y,s)dyds\Big|_{L^{2}}\\
&=\Big|\delta^{\varepsilon}\bar{U}'(x,t)-\int G(x,t;y)(\delta^{\varepsilon}\bar{U}')(y,0)dy\Big|_{L^{2}}\nonumber\\
&\leq \Big|\delta^{\varepsilon}\bar{U}'(x,t)-\bar{U}'(x_1)\int
g(\tilde{x}-\tilde{y},t)\delta^{\varepsilon}(\tilde{y},0)d\tilde{y}\Big|_{L^{2}}+C\zeta_0
t^{-(d-1)/4-1/2}\nonumber\\
&= C\zeta_0 t^{-(d-1)/4-1/2}.\nonumber
\end{align}
Combining this with the short time bound, we obtain
\begin{align}
&\Big|\int_0^{t}\int G(x,t-s;y)(\partial_s-L)\delta^{\varepsilon}\bar{U}'(y,s)dyds\Big|_{L^{2}}\\
&\leq C\zeta_0 (1+t)^{-(d-1)/4-1/2}.\nonumber
\end{align}




We now establish that
\begin{equation}
|II|_{L^{p}}\leq C\zeta_0 (1+t)^{-((d-1)/2)(1-1/p)-1/2}.
\end{equation}
Using \eqref{g22}, \eqref{deltabound}, and \eqref{5.21}, together
with the definition of $\zeta(t)$, we have
\begin{align}\label{IIa}
|II_a|_{L^{2}}&=\Big| \int_0^{t}\int
\partial_{x_j}G^{I}(x,t-s;y)N^{j}(y,s)dyds\Big|_{L^{2}}\\
&\leq C \int_0^{t} (1+t-s)^{-(d-1)/4-1/2} |U|_{L^{2}}^{2}(s)
ds\nonumber\\
&\leq C \zeta^{2}(t)
\int_0^{t}(1+t-s)^{-(d-1)/4-1/2}(1+s)^{-(d-1)/2-1+2\sigma}
ds\nonumber\\
&\leq  C\zeta^{2}(t)(1+t)^{-(d-1)/4-1/2+\sigma} \nonumber
\end{align}
and
\begin{align}\label{IIb}
|II_b|_{L^{2}}&=\Big| \int_0^{t}\int
\partial_{x_j}G^{I}(x,t-s;y)S^{j}(y,s)dyds\Big|_{L^{2}}\\
&\leq C\int_0^{t}(1+t-s)^{-(d-1)/4-1/2}|S^{j}|_{L^{1}}ds \nonumber\\
&\leq C\zeta_{0} \int_0^{t} (1+t-s)^{-(d-1)/4-1/2}
|\delta^{\varepsilon}|_{L^{2}}(s) |U|_{L^{2}}(s) ds\nonumber\\
&\leq C\zeta_0\zeta(t)(1+t)^{-(d-1)/4-1/2+\sigma}. \nonumber
\end{align}

The last inequality is true due to the following calculations;
\begin{align}
&\int_0^{t/2}(1+t-s)^{-(d-1)/4-1/2}(1+s)^{-(d-1)/4}(1+s)^{-(d-1)/4-1/2+\sigma}ds\\
&\leq (1+t/2)^{-(d-1)/4-1/2}\int_0^{t/2} (1+s)^{-(d-1)/2-1/2+\sigma}ds\nonumber\\
&\leq (1+t/2)^{-(d-1)/4-1/2}\big|(1+t/2)^{-(d-1)/2+1/2+\sigma}-1\big|\nonumber\\
&\leq C (1+t)^{-(d-1)/4-1/2+\sigma}\big|(1+t)^{-(d-1)/2+1/2}-(1+t)^{-\sigma}\big|\nonumber\\
&\leq C (1+t)^{-(d-1)/4-1/2+\sigma}\nonumber
\end{align}
It is true with arbitrary small $\sigma>0$ for $d=2$ and $\sigma=0$
for $d\geq 3$.

\begin{align}
&\int_{t/2}^{t}(1+t-s)^{-(d-1)/4-1/2}(1+s)^{-(d-1)/4}(1+s)^{-(d-1)/4-1/2+\sigma}ds\\
&\leq (1+t)^{-(d-1)/4-1/2+\sigma}(1+t)^{-(d-1)/4} \int_{t/2}^{t}(1+t-s)^{-(d-1)/4-1/2}ds \nonumber\\
&\leq (1+t)^{-(d-1)/4-1/2+\sigma}(1+t)^{-(d-1)/4}\int_{t/2}^{t}(1+t-s)^{-(d-1)/4-1/2+\epsilon}ds \nonumber\\
&\leq (1+t)^{-(d-1)/4-1/2+\sigma}(1+t)^{-(d-1)/2+1/2+\epsilon}\nonumber\\
&\leq (1+t)^{-(d-1)/4-1/2+\sigma}\nonumber.
\end{align}
The last inequality is true if we choose arbitrary small
$\epsilon>0$ for $d=3$ so that $-(d-1)/4-1/2+\epsilon\neq -1$ and
choose $\epsilon=0$ otherwise. \smallbreak Similarly, using
\eqref{g22}, \eqref{deltabound}, and \eqref{5.20},
\begin{align}\label{IIc}
|II_c|_{L^{2}}&=\Big|\int_0^{t}\int \partial_{x_1}G^{I}(x,t-s;y)R(y,s)dy ds \Big|_{L^{2}(x)}\\
&\leq
C\int_0^{t}(1+t-s)^{-(d-1)/4-1/2}|R|_{L^{1}}(s)ds\nonumber\\
&\leq C\int_0^{t}(1+t-s)^{-(d-1)/4-1/2}|\delta^{\varepsilon}|_{L^{2}}(s)|\delta_{x_j}^{\varepsilon}|_{L^{2}}(s)\nonumber\\
&\leq C\zeta_0^{2}\int_0^{t}(1+t-s)^{-(d-1)/4-1/2}(1+s)^{-(d-1)/2-1/2}ds\nonumber\\
&\leq C\zeta_0^{2} (1+t)^{-(d-1)/4-1/2}. \nonumber
\end{align}

By \eqref{IIa}, \eqref{IIb} and \eqref{IIc}, we have
\begin{equation}\label{IIre}
|II|_{L^{2}}\leq
C(\zeta^{2}(t)+\zeta_0\zeta(t)+\zeta_0^{2})(1+t)^{-(d-1)/4-1/2+\sigma}.
\end{equation}

We now establish that
\begin{equation}
|III|_{L^{2}}\leq
C(\zeta_0^{2}+\zeta_0\zeta(t)+\zeta^{2}(t))(1+t)^{-(d-1)/4-1/2+\sigma}.
\end{equation}

By \eqref{hfestimate}, \eqref{pro}, and \eqref{5.20}--\eqref{5.21},
we have
\begin{align}
&\Big|\int_0^{t}\int G^{II}(x,t-s;y)( N^{j}_{y_j}+ S^{j}_{y_j}+R_{y_1}) dyds\Big|_{L^{2}}\\
&\leq \int_0^{t} e^{-\theta (t-s)} |N^{j}_{x_j}+S^{j}_{x_j}+R_{x_1}|_{H^{3}}ds\nonumber\\
&\leq \int_0^{t} e^{-\theta (t-s)} |N^{j}+S^{j}+R|_{H^{4}}ds\nonumber\\
&\leq C\int_0^{t} e^{-\theta (t-s)}(|U|_{L^{\infty}}|U|_{H^{4}}+|\delta^{\varepsilon}|_{L^{\infty}}|U|_{H^{4}}+|\delta^{\varepsilon}|_{L^{\infty}}|\nabla_{\tilde{x}}\delta^{\varepsilon}|_{H^{4}})ds\nonumber\\
&\leq
C(\zeta_0^{2}+\zeta_0\zeta(t)+\zeta^{2}(t))(1+t)^{-(d-1)/4-1/2+\sigma}\nonumber
\end{align}
as long as $|U|_{H^{s}}$ and thus $|U|_{W^{1,\infty}}$ remains
sufficiently small. We shall verify in a moment that it indeed
remains small. \smallbreak Similarly, we can establish the estimates
for $IV$ and $V$. By the same proof as the one for $II$, together
with \eqref{fastpart}, we have
\begin{equation}
|IV|_{L^{2}}\leq C
(\zeta^{2}(t)+\zeta_0\zeta(t)+\zeta_0^{2})(1+t)^{-(d-1)/4-1/2+\sigma}.
\end{equation}

By an identical calculation as for the $III$ term except for the
fact that $V$ has one less derivatives, we can get
\begin{equation}
|V|_{L^{2}}(t)\leq
C(\zeta^{2}(t)+\zeta_0\zeta(t)+\zeta_0^{2})(1+t)^{-(d-1)/4-1/2+\sigma}.
\end{equation}
Therefore,
\begin{equation}
|U|_{L^{2}}(t)\leq
C_1(\zeta^{2}(t)+\zeta_0\zeta(t)+\zeta_0^{2})(1+t)^{-(d-1)/4-1/2+\sigma}.
\end{equation}
We have the desired inequality,
\begin{equation}\label{desired}
\zeta(t)\leq C_1(\zeta^{2}(t)+\zeta_0\zeta(t)+\zeta_0^{2})\leq
C_2(\zeta_0+\zeta^{2}(t))
\end{equation}
so long as $|U|_{H^{s}}\leq \varepsilon$ sufficiently small.
\smallbreak To complete the proof, we show that $|U|_{H^{s}}(t)\leq
\varepsilon$ remains small for all $t\geq 0$ if we choose
$|U|_{H^{s}}(0)\leq \zeta_0$ small enough, by the continuation
argument. By local well-posedness, for sufficiently small
$\varepsilon>0$, we can define
\begin{equation}\label{timeT} T:=\text{ sup }\{ \tau>0 :
|U|_{H^{s}}(\tau) <\varepsilon\}>0.
\end{equation}
As shown above, we have
\begin{equation}
\zeta(t)\leq C_2(\zeta_0+\zeta^{2}(t)) \text{ for } t\in[0,T),
\end{equation} which implies
\begin{equation}
\zeta(t)\leq 2C_2\zeta_0 \text{ for } t\in[0,T).
\end{equation}
Using \eqref{pro} again, we have, by choosing $\zeta_0$ so small
that there holds
\begin{equation}|U|_{H^{s}}(T)\leq e^{-\theta
T}\zeta_0 + \frac{2C_2^{2}\zeta_0}{\theta}<\varepsilon.
\end{equation}
By continuity of $H^{s}$-norm, there exists $h>0$ such that
$|U|_{H^{s}}(t)<\varepsilon$ for $t\in[0,T+h)$, which contradicts to
the definition of $T$. So, $T=\infty$, i.e.,
$|U|_{H^{s}}(t)<\varepsilon$ for $t\in[0,\infty)$. Thus, the desired
inequality $\eqref{desired}$ is true for all $t\geq 0$, which
implies that
\begin{equation}
\zeta(t)\leq 2C_2\zeta_0 \text{ for all } t\in[0,\infty).
\end{equation}
This completes the proof.
\end{proof}

\begin{theo}If $|\tilde{U}_0-\bar{U}|_{H^{[d/2]+5}}\leq \zeta_0$ sufficiently small, then there holds
\begin{equation}
\big|\tilde{U}(x,t)-\bar{U}(x_1-\delta^{\varepsilon}(\tilde{x},t))\big|_{L^{\infty}(x)}\leq
C\zeta_0 (1+t)^{-(d-1)/2-1/2}.
\end{equation}
\end{theo}
\begin{proof}
Using the expression \eqref{duhamel} similarly as in the previous
theorem, we get the $L^{\infty}$ bounds for each term.
\begin{equation}
|I|_{L^{\infty}}\leq C\zeta_0(1+t)^{-(d-1)/2-1/2}.
\end{equation}

Using \eqref{g22}, \eqref{deltabound} and \eqref{resultmain}, we
have
\begin{align}
|II_a|_{L^{\infty}}&\leq \int_0^{t} |G_{x_j}^{I}|_{L^{\infty}}(t-s)|U|_{L^{2}}^{2}(s)ds\nonumber\\
&\leq C\zeta^{2}_0\int_0^{t} (1+t-s)^{-(d-1)/2-1/2}(1+s)^{-(d-1)/2-1+2\sigma} ds\nonumber\\
&\leq C\zeta^{2}_0 (1+t)^{-(d-1)/2-1/2},\nonumber
\end{align}

\begin{align}
|II_b|_{L^{\infty}}&\leq \int_0^{t} |G_{x_j}^{I}|_{L^{\infty}}(t-s)|U|_{L^{2}}(s)|\delta^{\varepsilon}|_{L^{2}}(s)ds\nonumber\\
&\leq C\zeta_0 \int_0^{t} (1+t-s)^{-(d-1)/2-1/2}(1+s)^{-(d-1)/4-1/2+\sigma}(1+s)^{-(d-1)/4} ds\nonumber\\
&\leq C\zeta_0 (1+t)^{-(d-1)/2-1/2+\sigma}\nonumber
\end{align}
and
\begin{align}
|II_c|_{L^{\infty}}&\leq \int_0^{t} |G_{x_j}^{I}|_{L^{\infty}}(t-s)|\delta^{\varepsilon}_{x_j}|_{L^{2}}(s)|\delta^{\varepsilon}|_{L^{2}}(s)ds\nonumber\\
&\leq C\zeta^{2}_0 \int_0^{t} (1+t-s)^{-(d-1)/2-1/2}(1+s)^{-(d-1)/4-1/2}(1+s)^{-(d-1)/4} ds\nonumber\\
&\leq C\zeta^{2}_0 (1+t)^{-(d-1)/2-1/2}.\nonumber
\end{align}

We establish that
\begin{equation}\label{333}
|III|_{L^{\infty}}\leq C\zeta_0(1+t)^{-(d-1)/2-1/2+\sigma}.
\end{equation}
By \eqref{hfestimate}, we have
\begin{align}
\Big|\int_0^{t}\int G^{II}(x,t-s;y) N^{j}_{y_j}
dyds\Big|_{L^{\infty}}&\leq \Big|\int_0^{t}
\int G^{II}(x,t-s;y) N^{j}_{y_j} dyds\Big|_{H^{[d/2]+1}}\nonumber\\
&\leq \int_0^{t} e^{-\theta (t-s)} |N^{j}_{x_j}|_{H^{[d/2]+4}}ds\nonumber\\
&\leq C\int_0^{t} e^{-\theta (t-s)}|U|_{L^{\infty}}|U|_{H^{[d/2]+5}}ds\nonumber\\
&\leq C\zeta^{2}_0(1+t)^{-(d-1)/2-1+2\sigma}\nonumber,
\end{align}
\begin{align}
\Big|\int_0^{t}\int G^{II}(x,t-s;y) S^{j}_{y_j}
dyds\Big|_{L^{\infty}}&\leq
\Big|\int_0^{t}\int G^{II}(x,t-s;y) S^{j}_{y_j} dyds\Big|_{H^{[d/2]+1}}\nonumber\\
&\leq \int_0^{t} e^{-\theta (t-s)} |S^{j}_{x_j}|_{H^{[d/2]+4}}ds\nonumber\\
&\leq C\int_0^{t} e^{-\theta (t-s)}|\delta^{\varepsilon}|_{L^{\infty}}|U|_{H^{{[d/2]+5}}}ds\nonumber\\
&\leq C\zeta_0^{2}(1+t)^{-3(d-1)/4-1/2+\sigma}\nonumber,
\end{align} and
\begin{align}
\Big|\int_0^{t}\int G^{II}(x,t-s;y) R_{y_1}
dyds\Big|_{L^{\infty}}&\leq \Big|\int_0^{t}\int
G^{II}(x,t-s;y) R_{y_1} dyds\Big|_{H^{[d/2]+1}}\nonumber\\
&\leq
\int_0^{t} e^{-\theta (t-s)} |R_{x_1}|_{H^{[d/2]+4}}ds\nonumber\\
&\leq C\zeta_0^{2} (1+t)^{-(d-1)/2-5/2}\nonumber.
\end{align} Thus we have established \eqref{333}.
The identical calculation as in the estimation of $III$ gives the
desired bound for $V$, which is
\begin{equation}|V|_{L^{\infty}}\leq  C\zeta^{2}_0(1+t)^{-(d-1)/2-1/2+\sigma}.
\end{equation}
Therefore, we have
\begin{equation}
|U|_{L^{\infty}}\leq  C\zeta^{2}_0(1+t)^{-(d-1)/2-1/2+\sigma}.
\end{equation}

\end{proof}

\appendix
\section{Expansion of the Fourier symbol}
\label{appA}

{\it Low frequency expansion.}
We carry out the expansion of $P(\xi)$ in $\xi$ about zero.\\
\begin{equation}
P(\xi)=dQ-i\sum_{j=1}^{d}\xi_jA^{j}=\begin{pmatrix} 0 & 0\\ q_u &
q_v
\end{pmatrix}-i\sum_{j=1}^{d}\xi_j
\begin{pmatrix}
f^{j}_u&f^{j}_v\\
g^{j}_u&g^{j}_v
\end{pmatrix}
\end{equation}

\begin{claim}
To the second order, dispersion relations
\begin{equation}
\lambda(\xi)=\sigma(dQ-i\sum_{j=1}^{d}\xi_j A^{j}),\;\;\;
\lambda(0)=0
\end{equation}
\end{claim}
are given by
\begin{align}
\lambda(\xi)&=-i\xi\cdot a^{*}-\xi^{t} B^{*} \xi+\cdot\cdot\cdot
\;\;\;
\end{align} and
\begin{equation}
V(\xi)=V^{0}+\sum_{j=1}^{d}\xi_j V^{1}_{j}
+\sum_{j,k=1}^{d}\xi_j\xi_k V_{jk}^{2}+\cdot\cdot\cdot
\end{equation}
with $a=(a_{1},a_{2},...,a_{d})$ and $B^{*}=\Big[b^{*}_{jk}\Big]_{j,k=1}^{d}$, where
\begin{equation}
a^{*}_j=f^{j}_u-f^{j}_v q_v^{-1} q_u,
\end{equation}

\begin{align}
b_{jk}^{*}=
\left\{%
\begin{array}{ll}
 - f^{j}_v q_v^{-1} (g^{j}_u-g^{j}_v q_v^{-1} q_u-( f^{j}_u-f^{j}_v q_v^{-1}
q_u)q_v^{-1} q_u) & ,\text{ if } j=k \\
 -\frac{1}{2}\Big(f^{j}_v q_v^{-1}(g^{j}_u-g^{j}_v q_v^{-1}
q_u+(f^{k}_u-f^{k}_v q_v^{-1} q_u) q_v^{-1} q_u)\\
\;\;\;
+f^{k}_v q_v^{-1}(g^{k}_u-g^{k}_v q_v^{-1} q_u+(f^{j}_u-f^{j}_v q_v^{-1} q_u) q_v^{-1} q_u)\Big) & ,\text{ if  } j\neq k \\
\end{array}%
\right.,
\end{align}

\begin{equation}
V^{0}=\begin{pmatrix} 1 \\ -q_v^{-1}q_u \end{pmatrix},
\end{equation}
and
\begin{equation}
V^{1}_j=\begin{pmatrix} 1\\s^{1}_j\end{pmatrix}=\begin{pmatrix}
1\\-q_v^{-1}q_u +i q_v^{-1}(g^{j}_u-g^{j}_v q_v^{-1} q_u+a^{*}_j
q_v^{-1} q_u)\end{pmatrix}.\end{equation}

\begin{proof}
Set
\begin{align}
0&=\big(dQ-i\sum_{j=1}^{d}\xi_jA^{j}-\lambda(\xi)\big)V(\xi)\\
&=\big(dQ-i\sum_{j=1}^{d}\xi_jA^{j}+i\sum_{j=1}^{d}\xi_j
a^{*}_jI+\sum_{j,k=1}^{d}\xi_j\xi_k b_{jk}^{*}I+\cdot\cdot\cdot\big)\nonumber\\
&\;\;\;\;\;\;\;\;\;\;\;\;\;\;\;\;\;\;\;\;\;\;\times
\big(V^{0}+\sum_{j=1}^{d}\xi_j V^{1}_j +\sum_{j,k=1}^{d}\xi_j\xi_k
V_{jk}^{2}+\cdot\cdot\cdot\big).\nonumber
\end{align}
and let $V^{m}_j=\begin{pmatrix}r^{m}_j\\s^{m}_j\end{pmatrix}$ for
$m=0,1,2$. Collecting the 0th order term, we have
\begin{equation}
dQV^{0}=0.
\end{equation}
which yields
\begin{equation}V^{0}=\begin{pmatrix} r^{0} \\ -q_v^{-1}q_u
r^{0}\end{pmatrix}\end{equation}
 Collecting the 1st order terms, we have

\begin{equation}dQV^{1}_j-i(A^{j}- a_j^{*}I)V^{0}=0 \text{ for }j=1,2,...,d.\end{equation}

Examining the first coordinate, we have
\begin{equation}0=i(f^{j}_u-f^{j}_v q_v^{-1} q_u-a_j^{*}I)r^{0}\end{equation}
So, $r^{0}$ is the right eigenvector of $f^{j}_u-f^{j}_v q_v^{-1}
q_u$ corresponding to the eigenvalue $a^{*}_j$. Let $l^{0}$ be the
counterpart left-eigenvector. Then, the eigenvalue $a^{*}_j$ is
given by
\begin{equation}
a_j^{*}=l^{0}(f^{j}_u-f^{j}_v q_v^{-1} q_u)r^{0}.
\end{equation}

Examining the second coordinate equation, we have
\begin{equation}q_u r^{1}_j+q_v s^{1}_j=i(g^{j}_u-g^{j}_v q_v^{-1}
q_u+a^{*}_j q_v^{-1} q_u)r^{0},\end{equation} which yields
\begin{equation}s^{1}_j=-q_v^{-1}q_u r^{1}_j+
i q_v^{-1}(g^{j}_u-g^{j}_v q_v^{-1} q_u+a^{*}_j q_v^{-1}
q_u)r^{0}.\end{equation}
 So, $V^{1}_j=\begin{pmatrix}
r^{1}_j\\s^{1}_j\end{pmatrix}=\begin{pmatrix} r^{1}_j\\-q_v^{-1}q_u
r^{1}_j+i q_v^{-1}(g^{j}_u-g^{j}_v q_v^{-1} q_u+a^{*}_j q_v^{-1}
q_u)r^{0}\end{pmatrix}$.

\smallbreak Collecting 2nd order terms ($\xi_j\xi_k$ term),
\begin{equation}b^{*}_{jk}V^{0}+i(a_j^{*}I-A^{j})V^{1}_k+dQV^{2}_{jk}=0\end{equation}
For $j=k$, the first coordinate equation yields
\begin{align}
&b^{*}_{jj}r^{0}+i\Big((a^{*}_j-f^{j}_u)r^{1}_j+f^{j}_v\big(q_v^{-1}q_u
r^{1}_j-iq_v^{-1}(g^{j}_u-g^{j}_v q_v^{-1} q_u+a^{*}_j q_v^{-1}
q_u)r^{0}\big)\Big)
\\
&=b^{*}_{jj}r^{0}+i(a^{*}_j-f^{j}_u+f^{j}_v q_v^{-1}q_u
)r^{1}_j+f^{j}_v q_v^{-1}(g^{j}_u-g^{j}_v q_v^{-1} q_u+a^{*}_j
q_v^{-1} q_u)r^{0}=0\nonumber
\end{align}
For $j\neq k$, we have
\begin{equation}
b^{*}_{jk}V^{0}+i(a_j^{*}I-A^{j})V^{1}_k+dQV^{2}_{jk}+\beta^{*}_{kj}V^{0}+i(a_k^{*}I-A^{k})V^{1}_j+dQV^{2}_{kj}=0.
\end{equation}
The first coordinate equation gives
\begin{align}
&b^{*}_{jk}r^{0}+i(a^{*}_j-f^{j}_u+f^{j}_v q_v^{-1}q_u )r^{1}_k+f^{j}_v q_v^{-1}(g^{j}_u-g^{j}_v q_v^{-1} q_u+a^{*}_k q_v^{-1} q_u)r^{0}\\
&\;\;\;+b^{*}_{kj}r^{0}+i(a^{*}_k-f^{k}_u+f^{k}_v q_v^{-1}q_u
)r^{1}_j+f^{k}_v q_v^{-1}(g^{k}_u-g^{k}_v q_v^{-1} q_u+a^{*}_j
q_v^{-1} q_u)r^{0}
\nonumber\\
&= 2b^{*}_{jk}r^{0}+i(a^{*}_j-f^{j}_u+f^{j}_v q_v^{-1}q_u )r^{1}_k+f^{j}_v q_v^{-1}(g^{j}_u-g^{j}_v q_v^{-1} q_u+a^{*}_k q_v^{-1} q_u)r^{0}\nonumber\\
&\;\;\;+i(a^{*}_k-f^{k}_u+f^{k}_v q_v^{-1}q_u )r^{1}_j+f^{k}_v
q_v^{-1} (g^{k}_u-g^{k}_v q_v^{-1} q_u+a^{*}_j q_v^{-1}
q_u)r^{0}=0\nonumber
\end{align}
For the scalar case, i.e., $n=1$, we will denote, for simplicity,
$a_l^{*}=(a_{l1}^{*},...,a^{*}_{ld})$ by
$a^{*}=(a_1^{*},...,a^{*}_d)$. Then, we have
\begin{equation}
V^{0}=\begin{pmatrix} 1 \\ -q_v^{-1}q_u \end{pmatrix},
\end{equation}

\begin{equation}\label{ajstar}
a_j^{*}=f^{j}_u-f^{j}_v q_v^{-1} q_u,
\end{equation}

\begin{align}b^{*}_{jj}&=- f^{j}_v q_v^{-1}(g^{j}_u-g^{j}_v q_v^{-1} q_u-a^{*}_j
q_v^{-1} q_u)\label{bjj}\\
&=- f^{j}_v q_v^{-1} (g^{j}_u-g^{j}_v q_v^{-1} q_u-( f^{j}_u-f^{j}_v
q_v^{-1} q_u)q_v^{-1} q_u)\nonumber,
\end{align}
\begin{align}
b^{*}_{jk}&=-\frac{1}{2}\Big(f^{j}_v q_v^{-1}(g^{j}_u-g^{j}_v
q_v^{-1} q_u+a_k^{*} q_v^{-1} q_u) +f^{k}_v q_v^{-1}(g^{k}_u-g^{k}_v
q_v^{-1} q_u+a_j^{*} q_v^{-1} q_u)\Big)
\label{Abjk}\\
&=-\frac{1}{2}\Big(f^{j}_v q_v^{-1}(g^{j}_u-g^{j}_v q_v^{-1}
q_u+(f^{k}_u-f^{k}_v q_v^{-1} q_u) q_v^{-1} q_u)\nonumber\\
& \;\;\;\;\;\;\;\;\;\;\;\;\;\;\;\;\;\;\;\; +f^{k}_v
q_v^{-1}(g^{k}_u-g^{k}_v q_v^{-1} q_u+(f^{j}_u-f^{j}_v q_v^{-1} q_u)
q_v^{-1} q_u)\Big)\nonumber
\end{align}
and
\begin{equation}\label{Vj1}
V^{1}_j=\begin{pmatrix} 1\\s^{1}_j\end{pmatrix}=\begin{pmatrix}
1\\-q_v^{-1}q_u +i q_v^{-1}(g^{j}_u-g^{j}_v q_v^{-1} q_u+a^{*}_j
q_v^{-1} q_u)\end{pmatrix}.
\end{equation}
\end{proof}
Moreover, we let $\mathcal{B}^{*}$ be the viscosity matrix
$[b_{jk}^{*}]$ as in \eqref{bjk} and write
\begin{equation}\label{BBstar}
\mathcal{B}^{*}=b^{*}_{11}\begin{pmatrix} 1 & -b^{* t}\\
b^{*} & B^{*}
\end{pmatrix}
\end{equation}
where $b^{*}\in \RR^{d-1}$ and $B^{*}\in \RR^{(d-1)\times(d-1)}.$

\section{Asymptotic ODE: gap and conjugation lemmas}
\label{appB}

Consider a general family of first-order ODE
\begin{equation}
\label{gfirstorder} \WW'-{\mathbb A}(x_1, \Lambda)\WW=\FF
\end{equation}
indexed by a spectral parameter $\Lambda \in \Omega \subset \CC^m$,
where $W\in \CC^N$, $x_1\in \RR$ and ``$'$'' denotes $d/dx_1$.
\begin{ass}\label{h0}{$\,$}
\medbreak \textup{(h0) } Coefficient ${\mathbb A}(\cdot,\Lambda)$,
considered as a function from $\Omega$ into
$C^0(x_1)$ is analytic in $\Lambda$. Moreover, ${\mathbb A}(\cdot,
\Lambda)$ approaches exponentially to limits $\mA_\pm$ as $x_1\to
\pm \infty$, with uniform exponential decay estimates
\begin{equation}
\label{expdecay2} |(\partial/\partial x_1)^k(\mA- \mA_\pm)| \le
C_1e^{-\theta|x_1|/C_2}, \, \quad \text{\rm for } x_1\gtrless 0, \,
0\le k\le K,
\end{equation}
$C_j$, $\theta>0$, on compact subsets of $\Omega $.
\end{ass}
\medbreak

\medbreak

\begin{lem}[{The gap lemma [KS, GZ, ZH]}]
\label{gaplemma} Consider the homogeneous version $\FF\equiv 0$ of
\eqref{gfirstorder}, under assumption (h0). If $V^-(\Lambda)$ is an
eigenvector of $\mA_-$ with eigenvalue $\mu(\Lambda)$, both analytic
in $\Lambda$, then there exists a solution of \eqref{gfirstorder} of
form
\begin{equation}
 \WW(x_1, \Lambda) = V (x_1,\Lambda ) e^{\mu(\Lambda) x_1},
\end{equation}
where $V$ is $C^{1}$ in $x_1$ and locally analytic in $\Lambda$ and,
for any fixed $\btheta < \theta$, satisfies
\begin{equation}
\label{3.6g} V(x_1,\Lambda )=  V^-(\Lambda ) + \bfO (e^{-\bar
\theta|x_1|}|V^- (\Lambda)|),\quad x_1 < 0.
\end{equation}
\end{lem}

\begin{lem}
(The conjugation lemma). Given (h0), there exist locally to any
given $\Gamma_0 \in \Omega$ invertible linear transformations
$P_+(x_1,\Gamma)=I+\Theta_+(x_1,\Gamma)$ and
$P_-(x_1,\Gamma)=I+\Theta_-(x_1,\Gamma)$ defined on $x_1\geq 0$ and
$x_1\leq 0$, respectively, $\Phi_\pm$ analytic in $\Gamma$ as
functions from $\Omega$ to $\mathcal{C}^{0}[0,\pm\infty)$, such
that: \medbreak (i) For any fixed $0<\btheta<\theta$ and $0\leq k
\leq K+1$, $j\geq 0$,
\begin{equation}
\label{Pdecay} |(\partial/\partial \Lambda)^j(\partial/\partial
x_1)^k \Theta_\pm |\le C(j) C_1 C_2 e^{-\theta |x_1|/C_2} \quad
\text{\rm for } x_1\gtrless 0.
\end{equation}
\smallbreak (ii)  The change of coordinates $\WW=:P_\pm \ZZ$, $\FF=:
P_\pm \GG$ reduces \eqref{gfirstorder} to
\begin{equation}
\label{glimit} \ZZ'-\mA_\pm \ZZ = \GG \quad \text{\rm for }
x_1\gtrless 0.
\end{equation}
\end{lem}

Equivalently, solutions of \eqref{gfirstorder} may be factored as
\begin{equation}
\label{Wfactor} \WW=(I+ \Theta_\pm)\ZZ_\pm,
\end{equation}
where $\ZZ_\pm$ satisfy the limiting, constant-coefficient equations
\eqref{glimit} and $\Theta_\pm$ satisfy bounds \eqref{Pdecay}.
\begin{exam}
Consider the linearized equations:

\begin{equation}
U_t=LU:=-\sum_{j=1}^{d}(A^{j}U)_{x_j}+QU
\end{equation}

\begin{equation}
\hat{U}_t=L_{\tilde{\xi}}\hat{U}:=-(A^{1}\hat{U})'-\sum_{j=2}^{d}i\xi_jA^{j}\hat{U}+Q\hat{U}
\end{equation}
Consider a non-homogeneous eigenvalue problem:
\begin{equation}\label{nev}
(L_{\tilde{\xi}}-\lambda)W=f
\end{equation}
Eq. \eqref{nev} can be expressed in the form:
\begin{eqnarray}
W'&=&-(A^{1})^{-1}\big((A^{1})'+i\sum_{j=2}^{d}\xi_jA^{j}-Q+\lambda I\big)W-(A^{1})^{-1}f\nonumber\\
&=&{\mathbb A}(x_1, \Lambda)W-\FF\nonumber\\
\nonumber
\end{eqnarray}

\end{exam}

\section{Series expansion of the top eigenvalue of $L_{\tilde \xi}$}
\label{appC} Consider
\begin{equation}
L_{\tilde{\xi}} U =-({\bar{A}}^{1}U)' -i \sum_ {j\neq 1}
\xi_j{\bar{A}}^{j}U +\bar{Q}U.
\end{equation}
It can be shown that there exists a unique, analytic eigenvalue
\begin{equation}\label{evalueexp}
\lambda_0(\tilde{\xi})= 0+ \tilde{\gamma}^{1} \cdot \tilde{\xi} +
\tilde{\xi}^{t}\cdot(\tilde{\gamma}^{2}
\tilde{\xi})+\mathcal{O}(|\tilde{\xi}|^{3}),
\end{equation}
of $L_{\tilde{\xi}}$ perturbing from the top eigenvalue $\lambda=0$
of the operator $L_0$, with associated analytic right and left
eigenfunctions
\begin{align}\label{rightevectorexp}
\varphi(\tilde{\xi})&=\varphi^{0}+\varphi^{1}\cdot\tilde{\xi}+\tilde{\xi}^{t}\varphi^{2}\tilde{\xi}
+\mathcal{O}(|\tilde{\xi}|^{3})
\end{align}
and \begin{equation}\label{leftevectorexp}
 \pi(\tilde{\xi})=\pi^{0}+\pi^{1}\cdot\tilde{\xi}+\tilde{\xi}^{t}\pi^{2}\tilde{\xi}+\mathcal{O}(|\tilde{\xi}|^{3}).
\end{equation}
with $\varphi^{0}=\bar{U}'$ and $\pi^{0}=([u]^{-1},0)$.

\begin{lem}
The expansions \eqref{evalueexp}, \eqref{rightevectorexp} and
\eqref{leftevectorexp} hold with
$i\tilde{\gamma}^{1}=\bar{\tilde{a}},\bar{\tilde{a}}\in \RR^{d-1},$
and $\tilde{\beta}\in \RR^{(d-1)\times(d-1)}$, where
$-\tilde{\gamma}^{2}=\tilde{\beta}$ is positive definite. Here
\begin{equation}\label{abars}
\bar{\tilde{a}}=(\bar{a}_2,...,\bar{a}_d)=-([f_*^{2}][u]^{-1},...,[f_*^{d}][u]^{-1})
\end{equation}
\end{lem}
\begin{proof}
Let $\varphi(\tilde{\xi})=\varphi_0 +\sum_{j\neq 1} \Phi^{j} \xi_j +
\sum_{j,k\neq 1} \Psi_{jk}\xi_j\xi_k
+\mathcal{O}(|\tilde{\xi}|^{3})$. \smallbreak The eigenvalue
equation $L_{\tilde{\xi}} \varphi=\lambda_0 \varphi$ leads to
\begin{align}
&-(\bar{A}^{1}\Phi^{j})'-iF^{j}(\bar{U})'+\bar{Q}\Phi^{j}=\gamma_j
\bar{U}' \text{ for } j=2,...,d.
\end{align}

 Integrating from $-\infty$ to $\infty$ both sides, we have
\begin{equation}
\int_{-\infty}^{\infty}-(\bar{A}^{1}\Phi^{j})'-iF(\bar{U})'+ (0 \;
d\bar{q})^{t}\Phi^{j}dx_1= \int_{-\infty}^{\infty} \gamma_j
\bar{U}'dx_1
\end{equation}
 which yields, by looking at
$u$ coordinate,
\begin{equation}
-i[F^{j}_{I}]=\gamma_j[U_I]\text{ for } j=2,...,d.
\end{equation}
 For $u$ scalar case, we have
\begin{equation}
\bar{a}_j=i\gamma_j=-\dfrac{[f_*^{j}]}{[u]}\text{ for } j=2,...,d.
\end{equation}
On the other hand, by \textup{($\mathcal{D}$3)} in Assumptions 1.2
and the analytic eigenvalue expansion \eqref{evalueexp},
$-\tilde{\gamma}^{2}=\tilde{\beta}$ is positive definite.

\end{proof}

\end{document}